\documentclass[a4paper, 10pt, DIV10, headinclude=false, footinclude=false]{scrarticle}
\usepackage[english]{babel}
\usepackage{lipsum}
\usepackage{amsfonts}
\usepackage{amssymb}
\usepackage{mathtools}
\usepackage{amsthm}
\usepackage{bm}
\usepackage{bbm}
\usepackage{graphicx}
\usepackage{xspace}
\usepackage{epstopdf}
\usepackage{algorithmic}
\usepackage{csquotes}
\usepackage{stmaryrd}
\usepackage{trimclip}
\usepackage{scalerel}
\usepackage{hyperref}
\usepackage{cleveref}
\usepackage[shortlabels]{enumitem}
\usepackage{todonotes}

\usepackage{amsopn}

\usepackage{tikz}
\usepackage{tikz-3dplot}
\usepackage{tikzscale}
\usepackage{tikzsymbols}
\usetikzlibrary{calc,positioning,shapes.multipart,patterns,shapes.arrows}
\usetikzlibrary{decorations.markings,arrows.meta}
\usetikzlibrary{arrows}
\usetikzlibrary{patterns.meta}
\usetikzlibrary{fadings}
\usetikzlibrary{matrix}
\usetikzlibrary{plotmarks}
\usetikzlibrary{external}
\providecommand{\tikzexternaloutputdirectory}{out}
\tikzset{
  external/system call={
    pdflatex \tikzexternalcheckshellescape
      -halt-on-error -interaction=batchmode
      -output-directory="\tikzexternaloutputdirectory"
      -jobname "\image" "\texsource"
  }
}
\usepackage{caption}
\usepackage{subcaption}
\usepackage{pgfplots}
\usepackage{pgfplotstable}
\pgfplotsset{compat=newest}
\pgfplotsset{plot coordinates/math parser=false}
\pgfplotsset{
    table/search path={tikz/data}
}
\usepgfplotslibrary{groupplots,dateplot}

\tikzfading[name=fade left,
  left color=transparent!0,right color=transparent!75]
\tikzfading[name=fade right,
  left color=transparent!75,right color=transparent!0]

\tikzset
{midarrow/.style={decoration={markings,mark=at position 0.5 with
			{\arrow[xshift=2pt]{Latex[length=4pt,#1]}}},postaction={decorate}}
}

\tikzset{>=stealth}

\newcommand{\cf}{cf.\xspace}
\newcommand{\ie}{i.e.\xspace}

\NewDocumentCommand{\dG}{s}{dG\IfBooleanF{#1}{\xspace}}

\NewDocumentCommand{\wellposed}{s}{well-posed\IfBooleanF{#1}{\xspace}}

\NewDocumentCommand{\Friedrichs}{s}{Friedrichs'\IfBooleanF{#1}{\xspace}}

\NewDocumentCommand{\Friedrichstitle}{s}{FRIEDRICHS'\IfBooleanF{#1}{\xspace}}

\NewDocumentCommand{\Maxwells}{s}{Maxwell's\IfBooleanF{#1}{\xspace}}

\NewDocumentCommand{\discontinuous}{s}{discontinuous\IfBooleanF{#1}{\xspace}}

\NewDocumentCommand{\Galerkin}{s}{Galerkin\IfBooleanF{#1}{\xspace}}

\NewDocumentCommand{\Lts}{s}{Local time-stepping\IfBooleanF{#1}{\xspace}}

\NewDocumentCommand{\Ltstitle}{s}{LOCAL TIME-INTEGRATION\IfBooleanF{#1}{\xspace}}

\NewDocumentCommand{\lts}{s}{local time-stepping\IfBooleanF{#1}{\xspace}}

\NewDocumentCommand{\lti}{s}{local time-integration\IfBooleanF{#1}{\xspace}}

\NewDocumentCommand{\Lti}{s}{Local time-integration\IfBooleanF{#1}{\xspace}}

\NewDocumentCommand{\righthandside}{s}{right-hand-side\IfBooleanF{#1}{\xspace}}

\NewDocumentCommand{\leapfrog}{s}{leapfrog\IfBooleanF{#1}{\xspace}}

\NewDocumentCommand{\Leapfrog}{s}{Leapfrog\IfBooleanF{#1}{\xspace}}

\NewDocumentCommand{\CN}{s}{Crank-Nicolson\IfBooleanF{#1}{\xspace}}

\NewDocumentCommand{\LFC}{s}{Leapfrog-Chebychev\IfBooleanF{#1}{\xspace}}

\NewDocumentCommand{\lfc}{s}{leapfrog-Chebychev\IfBooleanF{#1}{\xspace}}

\NewDocumentCommand{\CFL}{s}{CFL\IfBooleanF{#1}{\xspace}}

\NewDocumentCommand{\selfadj}{s}{self-adjoint\IfBooleanF{#1}{\xspace}}

\NewDocumentCommand{\pd}{s}{positive definite\IfBooleanF{#1}{\xspace}}

\NewDocumentCommand{\psd}{s}{positive semi-definite\IfBooleanF{#1}{\xspace}}

\NewDocumentCommand{\wrt}{s}{w.r.t.\IfBooleanF{#1}{\xspace}}

\NewDocumentCommand{\timestepsize}{s}{time stepsize\IfBooleanF{#1}{\xspace}}

\NewDocumentCommand{\timestepsizes}{s}{time stepsizes\IfBooleanF{#1}{\xspace}}

\newcommand{\Gauss}{Gau\ss\xspace}
\newcommand{\maxwell}{Maxwell's equations\xspace}

\newcommand{\Krylov}{Krylov subspace method\xspace}
\newcommand{\Krylovs}{Krylov subspace methods\xspace}
\newcommand{\Krylovsos}{Krylov subspace methods}
\newcommand{\rk}{Runge-Kutta method\xspace}
\newcommand{\rks}{Runge-Kutta methods\xspace}

\newcommand{\gcs}{Gau{\ss} collocation methods\xspace}

\newcommand{\matsqrt}[1]{{#1}^{1/2}}

\newcommand{\invsqrt}[1]{{#1}^{-1/2}}

\newcommand{\inv}[1]{{#1}^{-1}}

\newcommand{\real}{\mathbb{R}}

\newcommand{\realvec}[1]{{\real}^{#1}}

\newcommand{\realmat}[2]{{\real}^{#1 \times #2}}

\newcommand{\complex}{\mathbb{C}}

\newcommand{\complexvec}[1]{{\complex}^{#1}}

\newcommand{\complexmat}[2]{{\complex}^{#1 \times #2}}

\NewDocumentCommand{\derv}{O{} m}{\partial_{#2}^{#1}}

\newcommand{\bbN}{\mathbb{N}}

\newcommand{\da}{\coloneqq}

\NewDocumentCommand{\re}{o}{
  \IfValueTF{#1}{
    \operatorname{Re}(#1)
  }{
    \operatorname{Re}
  }
}

\NewDocumentCommand{\im}{o}{
  \IfValueTF{#1}{
    \operatorname{Im}(#1)
  }{
    \operatorname{Im}
  }
}

\newcommand{\tn}[1]{t_{#1}}
\newcommand{\ts}{\tau}

\renewcommand{\dim}{d}

\newcommand{\abs}[1]{\left\vert #1 \right\vert}

\NewDocumentCommand{\ip}{m o}{
  \IfValueTF{#2}{
    \bigl( #1 \bigr)_{#2}
  }{
    \bigl( #1 \bigr)
  }
}

\makeatletter
\def\mathcolor#1#{\@mathcolor{#1}}
\def\@mathcolor#1#2#3{%
  \protect\leavevmode
  \begingroup
    \color#1{#2}#3%
  \endgroupx
}
\makeatother

\definecolor{KITlilac}{RGB}{160,0,120}

\newcommand{\exE}{E}
\newcommand{\exH}{H}
\newcommand{\exJ}{J}

\newcommand{\dGdegree}{k}

\NewDocumentCommand{\dofs}{o}{
  \IfValueTF{#1}
    {N_{\sol[#1]}}
    {N}
}

\NewDocumentCommand{\projLtwo}{s}{
  \IfBooleanTF{#1}{\Pi_{\meshdiam}}{\pi_{\meshdiam}}
}

\NewDocumentCommand{\polydegmost}{O{}}{\mathbb{Q}_{\dim}^{#1}}

\NewDocumentCommand{\polytotaldegmost}{O{}}{\mathbb{P}_{\dim}^{#1}}

\NewDocumentCommand{\normal}{O{} m}{\mathbf{n}^{#2}_{#1}}

\NewDocumentCommand{\problemdim}{o}{m_{\IfValueT{#1}{\sol[#1]}}}

\newcommand{\domain}{\Omega}
\newcommand{\boundary}{\Gamma}

\NewDocumentCommand{\celldiam}{O{}}{h_{\cell}^{#1}}

\NewDocumentCommand{\facediam}{O{}}{h_{\face}^{#1}}

\NewDocumentCommand{\meshdiam}{O{}}{h^{#1}}

\NewDocumentCommand{\maxmeshdiam}{s O{}}{
  \IfBooleanTF{#1}{
    \max_{\cell \in \mesh} \celldiam
  }{
    h_{\text{max}}^{#2}
  }
}

\NewDocumentCommand{\minmeshdiam}{s O{}}{
  \IfBooleanTF{#1}{
    \min_{\cell \in \mesh} \celldiam
  }{
    h_{\text{min}}^{#2}
  }
}

\NewDocumentCommand{\meshdiamCoarse}{O{}}{
  \meshdiam[#1]_c
}

\NewDocumentCommand{\meshdiamFine}{O{}}{
  \meshdiam[#1]_f
}

\NewDocumentCommand{\maxmeshdiamCoarse}{s O{}}{
  \IfBooleanTF{#1}{
    \max_{\cell \in \meshCoarse} \celldiam
  }{
    \meshdiam[#2]_{c,\text{max}}
  }
}

\NewDocumentCommand{\minmeshdiamCoarse}{s O{}}{
  \IfBooleanTF{#1}{
    \min_{\cell \in \meshCoarse} \celldiam
  }{
    \meshdiam[#2]_{c,\text{min}}
  }
}

\NewDocumentCommand{\maxmeshdiamFine}{s O{}}{
  \IfBooleanTF{#1}{
    \max_{\cell \in \meshFine} \celldiam
  }{
    \meshdiam[#2]_{f,\text{max}}
  }
}

\NewDocumentCommand{\minmeshdiamFine}{s O{}}{
  \IfBooleanTF{#1}{
    \min_{\cell \in \meshFine} \celldiam
  }{
    \meshdiam[#2]_{f,\text{min}}
  }
}

\newcommand{\mesh}{\mathcal{T}_{\meshdiam}}
\newcommand{\cell}{K}

\newcommand{\face}{F}

\newcommand{\cutoffModified}{\chi_m}

\newcommand{\cutoffLeapfrog}{\chi_{\text{LF}}}

\newcommand{\meshCoarse}{\mathcal{T}_{\meshdiam,c}}
\newcommand{\meshFine}{\mathcal{T}_{\meshdiam,f}}

\NewDocumentCommand{\dgspace}{o}{\mathbb{V}_{\meshdiam}\IfValueT{#1}{^{\sol[#1]}}}

\NewDocumentCommand{\brokenpoly}{s o m}{\polydegmost[#3](\mesh)\IfBooleanF{#1}{^{\problemdim[#2]}}}

\NewDocumentCommand{\sobolev}{m m}{H^{#1}(#2)}

\NewDocumentCommand{\brokensobolev}{s o m}{\sobolev{#3}{\mesh}\IfBooleanF{#1}{^{\problemdim[#2]}}}

\NewDocumentCommand{\Ltwo}{o}{L^2\IfValueT{#1}{(#1)}}

\NewDocumentCommand{\Ltwovec}{o m}{L^2(#2)^{\problemdim[#1]}}

\NewDocumentCommand{\Linfty}{o}{L^{\infty}\IfValueT{#1}{(#1)}}

\NewDocumentCommand{\Linftymatrix}{o m}{L^{\infty}(#2)^{\problemdim[#1] \times \problemdim[#1]}}

\NewDocumentCommand{\sobolevinftymatrix}{o m m}{W^{#2, \infty}(#3)^{\problemdim[#1] \times \problemdim[#1]}}

\NewDocumentCommand{\brokensobolevinftymatrix}{o m}{\sobolevinftymatrix[#1]{#2}{\mesh}}

\NewDocumentCommand{\norm}{m O{}}{\Vert #1 \Vert_{#2}}

\NewDocumentCommand{\seminorm}{m O{}}{\vert #1 \vert_{#2}}

\NewDocumentCommand{\dgip}{o m}{\ip{#2}_{\domain}}%

\NewDocumentCommand{\dgnorm}{o m}{\norm{#2}[\domain]}%

\NewDocumentCommand{\opnorm}{s o O{} m}{
  \IfBooleanTF{#1}{
    \norm{#4}[\dgspace[#3]\leftarrow\dgspace[#2]]
  }{
    \norm{#4}[\dgspace[#2]\leftarrow\dgspace[#2]]
  }
}

\NewDocumentCommand{\grhs}{o}{
  \IfNoValueTF{#1}{
    g
  }{
    g_{\sol[#1]}
  }
}

\NewDocumentCommand{\rhs}{s o m}{
  \IfBooleanTF{#1}{
    \IfValueTF{#2}{
      g_{\sol[#2]}
    }{
      g
    }
  }{
    \IfNoValueTF{#2}{
      g_{\meshdiam}^{#3}
    }{
      g_{\sol[#2], \meshdiam}^{#3}
    }
  }
}

\NewDocumentCommand{\avggrhshn}{o m}{
  \IfNoValueTF{#1}{
    \overline{g}_{\meshdiam}^{#2}
  }{
    \overline{g}_{\sol[#1], \meshdiam}^{#2}
  }
}

\NewDocumentCommand{\sol}{o}{
  \IfNoValueTF{#1}{x}{
    \IfEqCase{#1}{
      {first}{u}
      {second}{v}
      {dummy}{w}
    }
  }
}

\NewDocumentCommand{\soltest}{o}{
  \IfNoValueTF{#1}{y}{
    \IfEqCase{#1}{
      {first}{\psi}
      {second}{\phi}
    }
  }
}

\NewDocumentCommand{\exsol}{o m}{\widehat{\sol[#1]}^{#2}}

\NewDocumentCommand{\solh}{o O{}}{%
  \IfValueTF{#1}{%
      \IfEqCase{#1}{%
          {first}{\sol[#1]_{\meshdiam}^{#2}}%
          {second}{\sol[#1]_{\meshdiam}^{#2}}%
          {dummy}{\sol[#1]_{\meshdiam}^{#2}}%
      }[\sol_{\meshdiam}^{#2}]%
  }{%
      \sol_{\meshdiam}%
  }%
}

\NewDocumentCommand{\avgsolh}{o m}{\overline{\sol[#1]}_{\meshdiam}^{#2}}

\NewDocumentCommand{\soltesth}{o}{\soltest[#1]_{\meshdiam}}

\NewDocumentCommand{\material}{o}{
  \IfValueTF{#1}
    {\mathcal{M}_{\sol[#1]}}
    {\mathcal{M}}
}

\NewDocumentCommand{\friedop}{o}{
  \IfValueTF{#1}
    {\mathcal{L}_{\sol[#1]}}
    {\mathcal{L}}
}

\NewDocumentCommand{\friedopadj}{o}{
  \IfValueTF{#1}
    {\mathcal{L}_{\sol[#1]}^{\ast}}
    {\mathcal{L}^{\ast}}
}

\NewDocumentCommand{\friedcoeff}{o m}{
  \IfValueTF{#1}
    {L_{#2, \sol[#1]}}
    {L_{#2}}
}

\NewDocumentCommand{\friedoph}{o}{
  \IfValueTF{#1}
    {\mathcal{L}_{\sol[#1],\meshdiam}}
    {\mathcal{L}_{\meshdiam}}
}

\NewDocumentCommand{\friedophleapfrog}{o}{
  \friedoph[#1]^{\text{LF}}
}

\NewDocumentCommand{\friedophexplicit}{o}{
  \friedoph[#1]^{\text{c}}
}

\NewDocumentCommand{\friedophmodified}{o}{
  \friedoph[#1]^{\text{f}}
}

\NewDocumentCommand{\friedbndop}{o}{
  \IfValueTF{#1}
    {\mathcal{L}_{\boundary, \sol[#1]}}
    {\mathcal{L}_{\boundary}}
}

\NewDocumentCommand{\friedbndfield}{o}{
  \IfValueTF{#1}
    {L_{\boundary, \sol[#1]}}
    {L_{\boundary}}
}

\NewDocumentCommand{\friedassbndop}{o}{
  \IfValueTF{#1}
    {\mathcal{L}_{\partial, \sol[#1]}}
    {\mathcal{L}_{\partial}}
}

\NewDocumentCommand{\friedassbndfield}{o m}{
  \IfValueTF{#1}
    {L_{\partial, \sol[#1]}^{#2}}
    {L_{\partial}^{#2}}
}

\NewDocumentCommand{\Id}{o}{
    \mathcal{I}_{\IfNoValueTF{#1}{}{\sol[#1]}}
}

\NewDocumentCommand{\Ih}{o}{
	\mathbf{I}_{\IfNoValueTF{#1}{}{\sol[#1]}}
}

\newcommand{\filfun}{\Psi}

\newcommand{\filfunhat}{\widehat{\filfun}}

\newcommand{\auxfilfun}{\Theta}

\newcommand{\auxfilfuncardi}{\varphi}

\NewDocumentCommand{\friedopsecondorder}{s o}{
  \IfBooleanTF{#1}{
    \mathcal{A}_{\IfValueT{#2}{\text{#2}}}
  }{
    \friedoph[second] 
    \IfValueT{#2}{
      \IfEqCase{#2}{
        {lf}{\cutoffLeapfrog}
        {f}{\cutoffModified}
      }
    } \friedoph[first]
  }
}

\NewDocumentCommand{\filfunophat}{s o}{
  \IfBooleanTF{#1}
    {\filfunhat(-\tau^2\friedopsecondorder[#2])}
    {\mathbf{\filfunhat}}
}

\NewDocumentCommand{\filfunop}{s o}{
  \IfBooleanTF{#1}
    {\filfun(-\tau^2\friedopsecondorder[#2])}
    {\mathbf{\filfun}}
}

\NewDocumentCommand{\auxfilfunop}{s o}{
  \IfBooleanTF{#1}
    {\auxfilfun(-\tau^2\friedopsecondorder[#2])}
    {\mathbf{\auxfilfun}}
}

\NewDocumentCommand{\perturbationop}{o}{
  \IfValueTF{#1}
    {\mathcal{D}_{\text{lf},\sol[#1]}}
    {\mathcal{D}_{\text{lf}}}
}

\NewDocumentCommand{\auxfilfuncardiop}{s o}{
  \IfBooleanTF{#1}
    {\auxfilfuncardi(-\tau^2\friedopsecondorder[#2])}
    {\boldsymbol{\auxfilfuncardi}}
}

\NewDocumentCommand{\genop}{s o o}{%
  \IfBooleanTF{#1}{%
    \widehat{\mathbf{R}}%
  }{%
    \mathbf{R}%
  }
  \IfValueT{#3}{%
    \IfEqCase{#3}{%
      {inv}{^{-1}}%
      {LF}{^\text{LF}}%
      {CN}{^\text{CN}}%
      {LTI}{^\text{LTI}}%
    }
  }%
  \IfValueT{#2}{%
    \IfEqCase{#2}{%
      {p}{_+}%
      {m}{_-}%
      {pm}{_\pm}%
      {inv}{^{-1}}%
      {LF}{^\text{LF}}%
      {CN}{^\text{CN}}%
      {LTI}{^\text{LTI}}%
    }%
  }%
}

\NewDocumentCommand{\idmatrix}{o}{
  \IfValueTF{#1}{
    \mathbf{I}_{#1}
  }{
    \mathbf{I}
  }
}

\NewDocumentCommand{\zeromatrix}{o}{
  \IfValueTF{#1}{
    \mathbf{0}_{#1}
  }{
    \mathbf{0}
  }
}

\NewDocumentCommand{\onevector}{o}{
  \IfValueTF{#1}{
    \mathbbm{1}_{#1}
  }{
    \mathbbm{1}
  }
}

\NewDocumentCommand{\massmatrix}{s o}{
  \IfBooleanTF{#1}{
    \IfValueTF{#2}{
      \mathbf{\widetilde M}_{\sol[#2]}
    }{
      \mathbf{\widetilde M}
    }
  }{
    \IfValueTF{#2}{
      \mathbf{M}_{\sol[#2]}
    }{
      \mathbf{M}
    }
  }
}

\NewDocumentCommand{\friedmatrix}{s o}{
  \IfBooleanTF{#1}{
    \IfValueTF{#2}{
      \mathbf{\widetilde L}_{\sol[#2]}
    }{
      \mathbf{\widetilde L}
    }
  }{
    \IfValueTF{#2}{
      \mathbf{L}_{\sol[#2]}
    }{
      \mathbf{L}
    }
  }
}

\NewDocumentCommand{\friedmatrixmodified}{o}{
  \friedmatrix[#1]^{f}
}

\NewDocumentCommand{\friedmatrixnonstiff}{o}{
  \friedmatrix[#1]^{c}
}

\NewDocumentCommand{\solvector}{s o m}{
  \IfBooleanTF{#1}{
    \mathbf{\sol[#2]}(#3)
  }{
    \mathbf{\sol[#2]}^{#3}
  }
}

\NewDocumentCommand{\soltestvector}{s o m}{
  \IfBooleanTF{#1}{
    \mathbf{\soltest[#2]}(#3)
  }{
    \mathbf{\soltest[#2]}^{#3}
  }
}

\NewDocumentCommand{\rhsvector}{s o m}{
  \IfValueTF{#2}{%
    \IfEqCase{#2}{%
      {first}{\mathbf{g}_{\sol[#2]}}%
      {second}{\mathbf{g}_{\sol[#2]}}%
      {dummy}{\mathbf{g}_{\sol[#2]}}%
    }[\mathbf{g}]%
    \IfBooleanTF{#1}{(#3)}{^{#3}}%
  }{%
    \mathbf{g}\IfBooleanTF{#1}{(#3)}{^{#3}}%
  }%
}

\NewDocumentCommand{\fullrhsvector}{s o}{
  \IfValueTF{#2}{
  \IfBooleanTF{#1}{
    \mathbf{f}(#2)
  }{
    \mathbf{f}^{#2}
  }}{
    \mathbf{f}
  }
}

\NewDocumentCommand{\scdordermatrixmodified}{s}{%
  \IfBooleanTF{#1}{%
    \friedmatrixmodified[second]\friedmatrixmodified[first]%
  }{%
    \mathbf{A}_{f}%
  }%
}

\NewDocumentCommand{\scdordermatrixnonstiff}{s}{%
  \IfBooleanTF{#1}{%
    \friedmatrixnonstiff[second]\friedmatrixnonstiff[first]%
  }{%
    \mathbf{A}_{c}%
  }%
}

\NewDocumentCommand{\scdordermatrixmodifiedmodified}{s}{%
  \IfBooleanTF{#1}{%
    \massmatrix[second]^{-1}\friedmatrixmodified[first]\massmatrix[first]^{-1}\friedmatrixmodified[second]%
  }{%
    \widehat{\mathbf{A}}_{f}%
  }%
}

\NewDocumentCommand{\rkstage}{s o m O{}}{%
  \IfBooleanTF{#1}{%
    \IfValueTF{#2}{%
      #4{\mathbf{Z}}_{\sol[#2], #3}%
    }{%
      #4{\mathbf{Z}}_{#3}%
    }%
  }{%
    \IfValueTF{#2}{%
      #4{\mathbf{X}}_{\sol[#2], #3}%
    }{%
      #4{\mathbf{X}}_{#3}%
    }
  }%
}

\NewDocumentCommand{\rkrhs}{s o m O{}}{%
  \IfBooleanTF{#1}{%
    \IfValueTF{#2}{%
      #4{\widetilde{\mathbf{G}}}_{\sol[#2], #3}%
    }{%
      #4{\widetilde{\mathbf{G}}}_{#3}%
    }%
  }{%
    \IfValueTF{#2}{%
      #4{\mathbf{G}}_{\sol[#2], #3}%
    }{%
      #4{\mathbf{G}}_{#3}%
    }%
  }
}

\NewDocumentCommand{\preconditioner}{s}{%
  \IfBooleanTF{#1}{%
    \widetilde{\mathbf{B}}%
  }{%
    \mathbf{B}%
  }%
}

\NewDocumentCommand{\systemmatrix}{s}{%
  \IfBooleanTF{#1}{%
    \widetilde{\mathbf{A}}%
  }{%
    \mathbf{A}%
  }%
}

\NewDocumentCommand{\systemsol}{s}{%
  \IfBooleanTF{#1}{%
    \widetilde{\mathbf{x}}%
  }{%
    \mathbf{x}%
  }%
}

\NewDocumentCommand{\systemrhs}{s}{%
  \IfBooleanTF{#1}{%
    \widetilde{\mathbf{b}}%
  }{%
    \mathbf{b}%
  }%
}

\NewDocumentCommand{\dummyvec}{s}{%
  \IfBooleanTF{#1}{%
    \widetilde{\mathbf{y}}%
  }{%
    \mathbf{y}%
  }%
}

\NewDocumentCommand{\fov}{s m}{%
  \IfBooleanTF{#1}{%
    \mathbfcal{F}(#2)%
  }{%
    \mathcal{F}(#2)%
  }%
}

\newcommand{\rkmatrix}{\mathbfcal{O}\!\!\iota}

\newcommand{\rkmatrixeig}[1]{\lambda_{#1}}

\newcommand{\rkstages}{s}

\NewDocumentCommand{\qmrsol}{s o}{
  \IfBooleanTF{#1}{
    \widetilde{\mathbf{x}}_{#2}
  }{
    \mathbf{x}_{#2}
  }
}

\NewDocumentCommand{\defectTrapez}{o m}{\delta_{\text{tr} \IfValueT{#1}{,\sol[#1]}}^{#2}}

\NewDocumentCommand{\err}{m o}{e_{\IfValueT{#2}{\sol[#2]}}^{#1}}

\NewDocumentCommand{\discerr}{m o}{e_{{\meshdiam} \IfValueT{#2}{, \sol[#2]}}^{#1}}

\NewDocumentCommand{\projerr}{o m}{e_{\pi\IfValueT{#1}{, \sol[#1]}}^{#2}}

\NewDocumentCommand{\projerrwithop}{o m O{}}{e_{\pi, \friedop[#1]#2}^{#3}}

\NewDocumentCommand{\defect}{o m}{\widehat{\delta}_{\meshdiam \IfValueT{#1}{,\sol[#1]}}^{#2}}

\NewDocumentCommand{\defectsmooth}{o m}{\widehat{\delta}_{\meshdiam \IfValueT{#1}{,\sol[#1]}, s}^{#2}}

\NewDocumentCommand{\defectfried}{o m}{\widehat{\delta}_{\meshdiam \IfValueT{#1}{,\sol[#1]}, \friedoph}^{#2}}

\NewDocumentCommand{\ipDiff}{o m}{\Delta_{\text{m}\IfValueT{#1}{, \sol[#1]}}(#2)}

\NewDocumentCommand{\bilDiff}{o m m}{\Delta_{{#2}\IfValueT{#1}{, \sol[#1]}}(#3)}

\NewDocumentCommand{\frieddom}{o}{
  \mathcal{D}(\friedop[#1])
}

\NewDocumentCommand{\graphspace}{o}{H(\friedop[#1])}

\NewDocumentCommand{\dgfrieddom}{o}{V_{\meshdiam}^{\friedop[#1]}}

\NewDocumentCommand{\CFriedop}{o}{C_{\friedop[#1]}}

\NewDocumentCommand{\CFriedInvIneq}{o}{C_{\text{inv}, \friedop \IfValueT{#1}{,\sol[#1]}}}

\NewDocumentCommand{\CFriedInvIneqCoarse}{o}{C_{\text{inv}, \friedop, \IfValueT{#1}{\sol[#1],} c}}

\NewDocumentCommand{\Capprox}{o}{C_{\pi, \friedop \IfValueT{#1}{,\sol[#1]}}}

\newcommand{\polynomials}[1]{\mathcal{P}_{#1}}

\newcommand{\minpm}{\min_{\substack{p_m \in \polynomials{m}\\p_m(0)=1}}}

\newcommand{\qmrsuperset}{\mathcal{S}}

\NewDocumentCommand{\maxfov}{s}{
  \IfBooleanTF{#1}{
    \max_{z \in \fov{\systemmatrix*}}
  }{
    \max_{z \in \fov{\systemmatrix}}
  }
}

\newcommand{\realpha}{\alpha_{\scriptscriptstyle R}}
\newcommand{\imalpha}{\alpha_{\scriptscriptstyle I}}

\ifpdf
  \DeclareGraphicsExtensions{.eps,.pdf,.png,.jpg}
\else
  \DeclareGraphicsExtensions{.eps}
\fi

\DeclareMathOperator{\diag}{diag}

\DeclareMathAlphabet\mathbfcal{OMS}{cmsy}{b}{n}

\makeatletter
\DeclareRobustCommand{\revshortto}{%
	\mathrel{\mathpalette\short@to\relax}%
}

\newcommand{\short@to}[2]{%
	\clipbox{0 0 {.5\width} 0}{$\m@th#1\vphantom{+}{\shortleftarrow }$\,}%
}
\makeatother
\newcommand{\generalMatrix}{\mathbf{K}}
\newcommand{\invGeneral}{\generalMatrix^{-1}}
\newcommand{\pmGeneral}{p_m(\generalMatrix)}
\newcommand{\pmGeneralrz}{p_m(\generalMatrix)\brz}

\newcommand{\irm}{\operatorname{i}}

\newcommand{\Bmhat}{\mathbf{\widehat{B}}}
\newcommand{\Dmhat}{\mathbf{\widehat{D}}}
\newcommand{\Cmhat}{\mathbf{\widehat{C}}}
\newcommand{\Pmhat}{\mathbf{\widehat{P}}}

\newcommand{\tausq}{\tau^2}  %

\newcommand{\At}{\widetilde{\mathbf{A}}}

\newcommand{\hf}{h_{f}}

\newcommand{\be}{\mathbf{e}}
\newcommand{\bu}{\mathbf{u}}
\newcommand{\bff}{\mathbf{f}}
\newcommand{\bx}{\mathbf{x}}
\newcommand{\by}{\mathbf{y}}

\newcommand{\bv}{\mathbf{v}}
\newcommand{\bw}{\mathbf{w}}
\newcommand{\zero}{\mathbf{0}}

\newcommand{\brm}{\mathbf{r}_m}
\newcommand{\bym}{\mathbf{y}_m}
\newcommand{\bxm}{\mathbf{x}_m}

\newcommand{\brz}{\mathbf{r}_0}
\newcommand{\bxz}{\mathbf{x}_0}

\newcommand{\but}{\widetilde{\mathbf{u}}}

\newcommand{\calPm}{\mathbb{P}_m}
\newcommand{\Kry}{\mathcal{K}_m(\generalMatrix,\brz)}

\newcommand{\Pm}{\mathbf{P}_m}

\newcommand{\Hmt}{\widetilde{\mathbf{H}}_{m}}
\newcommand{\Hmpt}{\widetilde{\mathbf{H}}^+_{m}}
\newcommand{\HmtH}{\widetilde{\mathbf{H}}^*_{m}}
\newcommand{\Dmo}{\mathbf{D}_{m+1}}

\newcommand{\rhoGeneral}{\rho_\generalMatrix}

\newcommand{\rhoepsAt}{\rho_{\At}}

\newcommand{\Wm}{\mathbf{W}_m}
\newcommand{\Wmo}{\mathbf{W}_{m+1}}

\newcommand{\ipi}[1]{\langle #1 \rangle}

\newcommand{\lark}{\mathbf{\Lambda}_{\mathcal{O}\!\!\iota}}
\newcommand{\Tinv}{\mathbf{T}^{-1}}
\newcommand{\Tbar}{\overline{\mathbf{T}}}

\newcommand{\calS}{\mathcal{S}}

\newcommand{\sstage}{$s$-stage\xspace}

\newcommand{\where}{\quad \text{where} \quad}
\newcommand{\andeq}{\quad \text{and} \quad}

\newcommand{\card}[1]{\text{card}\left( #1 \right)}

\newcommand{\Grtg}{{\Gamma}_{\gamma}^e}
\newcommand{\Git}{{\Gamma}^i_{\gamma}}

\newcommand{\Gtez}{{\Gamma}_{\zeta}^e}
\newcommand{\Gtiz}{{\Gamma}_{\zeta}^i}

\newcommand{\Gtea}{{\Gamma}_{\alpha}^e}

\newcommand{\PQMR}{pQMR\xspace}
\newcommand{\Qt}{{Q}}
\newcommand{\Rt}{{R}}

\newcommand{\cAt}{c_{\At}}

\numberwithin{equation}{section}

\newenvironment{abstr}[1]{ \vspace{.05in}\footnotesize
	\parindent .2in
	{\upshape\bfseries #1. }\ignorespaces}{\par\vspace{.1in}}
	
	\newenvironment{keywords}{\begin{abstr}{Key words}}{\end{abstr}}
	\newenvironment{AMS}{\begin{abstr}{AMS subject classifications}}{\end{abstr}}
	\makeatletter
\newcommand{\mylabel}[2]{#2\def\@currentlabel{#2}\label{#1}}
\makeatother

\allowdisplaybreaks[4]

\newtheorem{theorem}{Theorem}[section]
\newtheorem{lemma}[theorem]{Lemma}	
\newtheorem{corollary}[theorem]{Corollary}

\crefname{assumption}{assumption}{assumptions}
\Crefname{Assumption}{Assumption}{Assumptions}

\theoremstyle{remark}

\theoremstyle{definition}

\numberwithin{equation}{section} 

\usepackage{amsopn}

\begin{document}
\title{
	Efficient higher-order local time integration for Friedrichs' systems%
	\thanks{Funded by the Deutsche Forschungsgemeinschaft (DFG, German Research Foundation) --- Project-ID 258734477 --- SFB 1173}
}
\author{
  Marlis Hochbruck\thanks{Institute for Applied and Numerical Mathematics, Karlsruhe Institute of Technology, Germany}
  \and Jonas K\"{o}hler\footnotemark[2] \and Malik Scheifinger\footnotemark[2]
}
\date{\today}
\maketitle
\begin{abstract} 
	In this paper, we construct an efficient higher-order local time integration scheme for spatially discretized linear \Friedrichs systems. 
	In particular, our interest is in problems where only a few of the mesh elements are small while the majority of the elements is much larger. 
	The special combination of two methods like the \leapfrog method on the coarse part of the mesh and the \CN method on the fine part as was done in \cite{HocS16,HocK22} is not suitable for higher-order time integration.
	Therefore, we suggest to approximate the solution of the linear systems arising in each time step by a preconditioned Krylov subspace method, e.g., the quasi-minimal residual method by Freund and Nachtigal \cite{Freund_Nachtigal_1991}.
	The techniques developed here for linear problems also carry over to nonlinear problems, where linear systems of the same type arise within a Newton-type iteration.
	
	Motivated by the analysis of locally implicit methods by Hochbruck and Sturm \cite{HocS16}, we show how to construct a preconditioner in such a way that the number of iterations required by the Krylov subspace method to achieve a certain accuracy is bounded independently of the diameter of the small mesh elements. 
	We prove this behavior by using Faber polynomials and complex approximation theory.
	
	The cost to apply the preconditioner consists of the solution of a small linear system, whose dimension corresponds to the degrees of freedom within the fine part of the mesh (and its next coarse neighbors). 
	If this dimension is small compared to the size of the full mesh, the preconditioner is very efficient.
	
	We conclude by verifying our theoretical results with numerical experiments for the fourth-order Gau{\ss}-Legendre Runge--Kutta method.
\end{abstract}

  
\section{Introduction}
\Friedrichs systems play a crucial role in the modeling of various physical phenomena. Important examples are advection equations, acoustic wave equations, \maxwell, and Dirac equations; see \cite{MFO2023} for more details.
Though finite difference time-domain methods \cite{Yee_66} are still predominantly utilized to solve \maxwell, numerous other methods based on finite element or finite volume space discretizations have been introduced and are gaining more and more importance. 
    
The numerical solution of time-dependent partial differential equations by a method of lines approach involves first discretizing in space and then integrating the semi-discrete system in time.
For the space discretization, discontinuous Galerkin (dG) methods (see \cite{DiPE12} and references therein) are popular due to their flexibility in treating complex geometries and discontinuous material parameters. 
Since dG methods lead to block-diagonal mass matrices, explicit time integration schemes can be implemented very efficiently.
Unfortunately, explicit time integration methods are subject to the Courant--Friedrichs--Lewy (CFL) condition depending on the minimum diameter of all mesh elements, denoted by $\minmeshdiam$, that is, the time step $\ts$ needs to satisfy $\ts \lesssim \minmeshdiam$.  
Here, we are interested in locally refined meshes, where most of the mesh elements are coarse but a very small number of mesh elements are fine.
The latter require very small time steps on all mesh elements, which makes the computation inefficient.
An alternative is to use implicit time integrators which can eliminate the CFL condition completely but require solving a large linear system of equations in all degrees of freedom at each time step.
Unfortunately, this is expensive and might not even be feasible for large 3D problems.
  
To tackle this problem, locally implicit (LI) methods \cite{Ver11,DesLM13,HocS16,Chabassier_Imperiale_2016,DesLM17,HocS19} and local time-stepping (LTS) methods \cite{Pip06,DiaG09,GroM10,GroMM15} or more general local time integrators (LTI) from our previous work \cite{Hochbruck_Scheifinger_2026} were introduced and studied. 
While there is a rigorous analysis of LTI methods of order two for linear problems, it is not clear how to prove the stability for higher-order LTI methods constructed via composition methods \cite[Section II.4]{HaiLW06_book}.

In this paper, we introduce a new strategy to develop a higher-order time integration method to solve the general class of linear \Friedrichs systems possessing a two-field structure on a locally refined spatial grid in a computationally efficient way. 
For the time integration, we propose to use higher-order implicit \Gauss collocation methods \cite[Sec.~IV.5]{Hairer_Wanner_1996}. 
It is well known that these methods are A-stable and symplectic, and have maximum order $2\rkstages$ for nonstiff problems. Moreover,
for all $\rkstages$, rigorous error bounds are available for a class of problems which includes \Friedrichs systems, cf.\ \cite[Section~3.1]{Hochbruck_Pazur_2015}. This makes them the ideal candidates for our applications. Unfortunately, the analysis only holds for \rks which are not only A-stable but also algebraically stable and coercive, properties which are not satisfied for general implicit methods such as the computationally attractive DIRK methods.

The drawback of implicit schemes is that one has to solve large linear systems of equations in each time step.
For this, iterative solvers, especially (preconditioned) \Krylovs, are popular choices; see \cite{Saad_2003} and references therein.
For instance, block preconditioners for structure-preserving discretizations of \maxwell were proposed in \cite{Philips_Shadid_Cyr_18,Adler_Hu_Zik_17}. 
Other popular preconditioners include those based on multigrid \cite{Hiptmair_99,Lubomir_10,Aruliah_Ascher_02} and Schwarz domain decomposition \cite{Bonazzoli_Dolean_Graham_Spence_19,Cyr_Gander_Thomas_07}.
However, these preconditioning techniques have not yet been adapted to locally refined meshes.

We show that the linear systems of equations which arise in the implementation of \Gauss \rks for linear \Friedrichs systems can be transformed using techniques from \cite[Sec.~IV.5]{Hairer_Wanner_1996} in such a way that we obtain an equivalent linear system with a complex symmetric coefficient matrix.  Thus, we use the quasi-minimal residual (QMR) method to solve it iteratively \cite{Freund_92,Freund_Nachtigal_1991,Freund_Nachtigal_1994}.
Our main contribution is to construct a suitable preconditioner for
the QMR method and to prove that the number of iterations required to
reach a certain accuracy is independent of the fine mesh. Roughly
speaking, this preconditioner only acts on the fine mesh and is thus
very efficient if the number of fine mesh elements is small compared
to the total degrees of freedom.  Hence, the overall algorithm is
unconditionally stable, obeys rigorous error bounds which do not
deteriorate under mesh refinement,
and has linear complexity with a constant that only depends on the
coarse mesh.

The paper is organized as follows. 
In \Cref{sec:problem_setting}, we present our model problem and notation, and recall properties
of the spatial discretization of \Friedrichs systems using discontinuous Galerkin methods with central fluxes.
\Cref{sec:higher_order_RK} is dedicated to higher-order implicit \rks.
In \Cref{sec:preconditioning}, we recall known results on \Krylovsos, in particular error bounds. Our main contribution is to prove that the field of values of the preconditioned matrix is independent of the fine mesh and to combine this with error bounds for Krylov subspace methods.
Finally, in \Cref{sec:numerical_exp}, we verify our theoretical
findings with numerical experiments and compare the efficiency of the
higher-order locally implicit method with other local time integration
schemes \cite{HocS16,Hochbruck_Scheifinger_2026} and a low-storage Runge--Kutta method \cite{Busch_Diehl_Niegemann_2021}.

\section{Problem setting}\label{sec:problem_setting}
Let $\Omega \subset \realvec{d}, d=1,2,3,$ be an open, polygonal domain. 
The linear \Friedrichs system in a two field formulation reads
\begin{subequations}\label{eq:two-field-system}
\begin{align}
    \derv{t}\sol[first] 
    &= 
    \friedop[second]\sol[second] 
    + 
    \rhs*[first]{},
    & \Omega \times \real_+, \\
    \derv{t}\sol[second]
    &=
    \friedop[first]\sol[first] 
    +
    \rhs*[second]{},
    & \Omega \times \real_+, \\
    \sol[first](0) &= \sol[first]^0, \quad \sol[second](0) = \sol[second]^0, & \Omega,
\end{align}
where $\derv{t}$ denotes the partial derivative with respect to time and \(\rhs*[first]{} \colon \real_+ \to \Ltwovec[first]{\domain},\ \rhs*[second]{} \colon \real_+ \to \Ltwovec[second]{\domain}\) are forcing terms.
\end{subequations}
The \Friedrichs operators \(\friedop[first]\) and \(\friedop[second]\) are defined by
\begin{equation}\label{eq:two-field-friedops}
    \material[second]\friedop[first]\sol[first] = \sum_{i = 1}^{\dim} \friedcoeff{i}\partial_i\sol[first], 
    \qquad
    \material[first]\friedop[second]\sol[second] = \sum_{i = 1}^{\dim} \friedcoeff{i}^{\mathrm{T}}\partial_i\sol[second],
\end{equation}
with coefficients \(\friedcoeff{i}\in\real^{\problemdim[second]\times\problemdim[first]}\) for \(i = 0,\dots,\dim\). Here \(\partial_i\) denotes the partial derivative with respect to the \(i\)th spatial variable.
The material tensors \(\material[first]\in\Linftymatrix[first]{\domain}\), \(\material[second]\in\Linftymatrix[second]{\domain}\) are assumed to be symmetric and uniformly positive definite, and by \(\problemdim[first], \problemdim[second]\in\bbN\) we count the components of \(\sol[first]\) and \(\sol[second]\).
To make the system~\eqref{eq:two-field-system} \wellposed we have to prescribe boundary conditions. These are embodied into the domains \(\frieddom[first]\) and \(\frieddom[second]\), see~\cite{DiPE12,MFO2023} for details.
We assume that the boundary conditions do not introduce damping, \cf~\cite[Assump. 11.8]{MFO2023}, which leads to the adjointness property
\begin{equation}\label{eq:friedop-adjointness}
    \dgip{\friedop[first]\sol[first], \sol[second]} 
    = 
    - \dgip{ \sol[first],\friedop[second]\sol[second]}
    \qquad\text{for all } 
    \sol[first]\in\frieddom[first], \sol[second]\in\frieddom[second],
\end{equation}
in the weighted $\Ltwo$ inner products defined via
\begin{align}\label{eq:weighted-L2-ip}
    \dgip{\sol[first], \widetilde{u}} 
    &= 
    \ip{\material[first] \sol[first], \widetilde{u}}_{\Ltwo(\Omega)}, 
    \qquad \sol[first], \widetilde{u} \in \Ltwovec[first]{\Omega} .
\end{align}
The corresponding norm is denoted by $\dgnorm{\cdot}$. 
The inner product weighted by \(\material[second]\) is defined analogously, and it is always clear from the inserted argument which inner product or norm is meant.

For a full discretization of \eqref{eq:two-field-system}, we first discretize in space using a dG method with central fluxes on a suitable mesh $\mesh$ \cite{DiPE12},\cite[Section~4]{HocK20}. 
On this mesh we define the spaces 
\begin{subequations}
\begin{align}
    \dgspace[first] 
    &= 
    \{\sol[first]\in\Ltwo[\domain] \ \vert \ \sol[first]\vert_\cell\in\polytotaldegmost[\dGdegree](\cell) \text{ for all } \cell\in\mesh\}^{\problemdim[first]}, \\
    \dgspace[second] 
    &= 
    \{\sol[second]\in\Ltwo[\domain] \ \vert \ \sol[second]\vert_\cell\in\polytotaldegmost[\dGdegree](\cell) \text{ for all } \cell\in\mesh\}^{\problemdim[second]}, 
\end{align}
\end{subequations}
consisting of functions in \(\Ltwo(\domain)\) which are polynomials of total degree at most $\dGdegree$ on each mesh element. 
If it is clear by context which space we mean, we simply write \(\dgspace\).
The dG method then yields
\begin{subequations}\label{eq:semi-two-field-system}
	\begin{align}
	\derv{t}\solh[first]
    &= 
    \friedoph[second]\solh[second]
    + 
    \rhs[first]{}, & \real_+, \\
    \derv{t}\solh[second]
    &=
    \friedoph[first]\solh[first]
    +
    \rhs[second]{}, & \real_+, \\
    \solh[first](0) &= \solh[first][0], \quad \solh[second](0) = \solh[second][0], &
	\end{align}
\end{subequations}
where $\friedoph[first]$ and $\friedoph[second]$ are spatially discretized \Friedrichs operators, and $\solh[first][0]$, $\solh[second][0]$, and $\rhs[first]{}$, $\rhs[second]{}$ are $\Ltwo$ projections of $\sol[first]^0$, $\sol[second]^0$ and $\rhs*[first]{}$, $\rhs*[second]{}$ respectively, onto $\dgspace[first], \dgspace[second]$ with respect to the weighted inner products defined in \eqref{eq:weighted-L2-ip}.
We refer to \cite{DiPE12,HocK20} for details on the dG discretization.

For \(\sol[dummy] \in \{\sol[first], \sol[second]\}\), let $\{\phi_1^{\sol[dummy]},\dots,\phi_{\dofs[dummy]}^{\sol[dummy]}\}$ be a basis of $\dgspace[dummy]$. 
Then an element \(\solh \colon (0,T) \to \dgspace[dummy]\) can be represented as
\begin{equation*}
    \solh(t) = \sum_{j=1}^{N} \sol_j(t) \phi_j^{\sol[dummy]},
\end{equation*}
with coefficient vector \(\solvector*{t} = (\sol_j(t))_{j=1}^{\dofs[dummy]}\) .
This results in mass and stiffness matrices given by
\begin{subequations} \label{eq:mass-stiffness}
    \begin{align}
        (\massmatrix[first])_{l,j} &= \dgip{\phi_j^{\sol[first]},\phi_l^{\sol[first]}}, &
        (\friedmatrix*[first])_{l,j} &= \dgip{\friedoph[first]\phi_j^{\sol[first]},\phi_l^{\sol[second]}}, \\
        (\massmatrix[second])_{l,j} &= \dgip{\phi_j^{\sol[second]},\phi_l^{\sol[second]}}, &
        (\friedmatrix*[second])_{l,j} &= \dgip{\friedoph[second]\phi_j^{\sol[second]},\phi_l^{\sol[first]}}.
    \end{align}
\end{subequations}
Then, \eqref{eq:semi-two-field-system} is equivalent to the following
system of ordinary differential equations,
\begin{equation}\label{eq:matrix-two-field-system}
	\begin{aligned}
		\derv{t} \solvector[first]{} 
        &= 
        \friedmatrix[second]\solvector[second]{} + \rhsvector[first]{},  
        & \friedmatrix[second] = \massmatrix[first]^{-1} \friedmatrix*[second], \\
        \derv{t} \solvector[second]{}
        &= 
        \friedmatrix[first]\solvector[first]{} + \rhsvector[second]{},
        & \friedmatrix[first] = \massmatrix[second]^{-1} \friedmatrix*[first], \\
        \solvector*[first]{0} &= \solvector[first]{0}, 
        \quad 
        \solvector*[second]{0} = \solvector[second]{0}. &
    \end{aligned}
\end{equation}
Here, $\solvector[first]{0}$ and $\solvector[second]{0}$ are the coefficient vectors of $\solh[first][0]$ and $\solh[second][0]$ respectively.

With an abuse of notation, given  $\solh, \soltesth \in \dgspace[dummy]$, \(\sol[dummy] \in \{\sol[first], \sol[second]\}\), with coefficient vectors  $\solvector{}, \soltestvector{} \in \complexvec{\dofs[dummy]}$, we define
\begin{equation}\label{def:weighted_ip}
	\dgip{\solvector{}, \soltestvector{}}
    \da 
    \soltestvector{\ast} \massmatrix[dummy] \solvector{} = \dgip{\solh, \soltesth},
\end{equation}
and do so analogously for the induced norms in $\complexvec{\dofs[dummy]}$. Here, $\ast$ denotes the conjugate transpose. 
For the matrix norms, we also take these weights into account, since then these norms are equivalent to the operator norms of the discrete operators $\friedoph[first]$ and $\friedoph[second]$, \ie, we have
\begin{subequations}\label{eq:normweight}
    \begin{align}
        \opnorm*[first][second]{\friedmatrix[first]} 
        &=
        \sup_{\solvector{} \in \complexvec{\dofs[first]} \setminus \{\zero\}} \frac{
            \dgnorm{\friedmatrix[first]\solvector{}}
        }{
            \dgnorm{\solvector{}}
        }
        = 
        \sup_{\solh \in \dgspace[first] \setminus \{0\}} \frac{
            \dgnorm{\friedoph[first]\solh}
        }{
            \dgnorm{\solh}
        }
        = 
        \opnorm*[first][second]{\friedoph[first]}, \\
        \opnorm*[second][first]{\friedmatrix[second]} 
        &=
        \sup_{\solvector{} \in \complexvec{\dofs[second]} \setminus \{\zero\}} \frac{
            \dgnorm{\friedmatrix[second]\solvector{}}
        }{
            \dgnorm{\solvector{}}
        }
        = 
        \sup_{\solh \in \dgspace[second] \setminus \{0\}} \frac{
            \dgnorm{\friedoph[second]\solh}
        }{
            \dgnorm{\solh}
        }
        = 
        \opnorm*[second][first]{\friedoph[second]}.
    \end{align}
\end{subequations}
    
In this paper, we are interested in locally refined meshes. 
We refer to our earlier papers \cite{HocS16, HocS19, Hochbruck_Scheifinger_2026} for detailed explanations on these meshes, but to keep this paper self-contained, we introduce the necessary notation  here.	
A locally refined mesh is a mesh in which most of the mesh elements are coarse and very few mesh elements are fine. 
Let $\meshCoarse$ and $\meshFine$ denote the collection of all coarse and fine mesh elements, respectively. 
We denote by $\meshdiamFine$ and $\meshdiamCoarse$ the size of smallest mesh elements in $\meshFine$ and in $\meshCoarse$, respectively.
These two sets are related to each other via
\begin{equation*}\label{assumption:mesh_size}
    \meshdiamFine \ll \meshdiamCoarse \andeq \card{\meshFine}\ll \card{\meshCoarse}. 
\end{equation*}
Based on this decomposition of the mesh, the matrices defined in \eqref{eq:matrix-two-field-system}
can be split into
\begin{equation}\label{eq:split_def}
    \friedmatrix[first]
    =
    \friedmatrixmodified[first] + \friedmatrixnonstiff[first],
    \qquad
    \friedmatrix[second]
    =
    \friedmatrixmodified[second] + \friedmatrixnonstiff[second],
\end{equation}	
cf. \cite{HocS16,Hochbruck_Scheifinger_2026} for more details. 
The indices $f$ and $c$ indicate that the elements on which $\friedmatrixmodified[first],\friedmatrixmodified[second]$ act belong to the fine submesh \(\meshFine\) and the ones on which $\friedmatrixnonstiff[first],\friedmatrixnonstiff[second]$ act, belong to the coarse submesh \(\meshCoarse\).
In fact, it was shown in \cite{HocS16} that not only the fine elements have to be considered but also their direct coarse neighbors. 
        
Let us state some properties of these matrices which are inherited from their corresponding discrete operators, cf.\ \cite{HocS16}. First, $\friedmatrix[first]$ and $\friedmatrix[second]$ are adjoint to each other, i.e., for all $\solvector[first]{} \in \complexvec{\dofs[first]}, \solvector[second]{} \in \complexvec{\dofs[second]}$, it holds
\begin{equation}\label{eq:full_adjoint} 
	\dgip{\friedmatrix[first]\solvector[first]{}, \solvector[second]{}} 
    =
    -\dgip{\solvector[first]{}, \friedmatrix[second]\solvector[second]{}}.    
\end{equation}
It is easy to verify that these split matrices preserve the adjointness property of their respective full ones, i.e.,
\begin{equation}\label{eq:split_adjoint}
    \dgip{\friedmatrixnonstiff[first]\solvector[first]{}, \solvector[second]{}} 
    = 
    -\dgip{\solvector[first]{}, \friedmatrixnonstiff[second]\solvector[second]{}},
    \qquad
    \dgip{\friedmatrixmodified[first]\solvector[first]{}, \solvector[second]{}} 
    = 
    -\dgip{\solvector[first]{}, \friedmatrixmodified[second]\solvector[second]{}}.
\end{equation}	
In addition, they satisfy 
\begin{equation}\label{eq:split_product}
	\friedmatrixnonstiff[second]\friedmatrixnonstiff[first] 
    =
    \friedmatrixnonstiff[second]\friedmatrix[first],
    \qquad
    \friedmatrixmodified[second]\friedmatrixmodified[first] 
    =
    \friedmatrixmodified[second]\friedmatrix[first] .
\end{equation}	
Furthermore, combining the above properties, it holds
\begin{equation}\label{eq:split_norm}
    \dgnorm{\friedmatrix[first] \solvector[first]{}}^2
    =
    \dgnorm{\friedmatrixnonstiff[first] \solvector[first]{}}^2
    +
    \dgnorm{\friedmatrixmodified[first] \solvector[first]{}}^2.
\end{equation}	

One of the important results from \cite{HocS16} is that the explicit split matrices $\friedmatrixnonstiff[first]$ and $\friedmatrixnonstiff[second]$ can be bounded independently of the fine mesh, that is, using the weighted norms in \eqref{eq:normweight}, we have 
\begin{equation}
	\opnorm*[first][second]{\friedmatrixnonstiff[first]} \leq c \meshdiamCoarse[-1], 
    \qquad 
    \opnorm*[second][first]{\friedmatrixnonstiff[second]} \leq c \meshdiamCoarse[-1],
\end{equation}		
with a constant $c$ that is independent of $\meshdiamFine$ and $\meshdiamCoarse$.

In \cite{HocS16,HocS19}, these split matrices were constructed to develop local time integration methods.
Here, we use these split matrices in a different way, namely to construct preconditioners which improve the performance of \Krylovs.

\section{Higher-order implicit Runge--Kutta methods}\label{sec:higher_order_RK}
In this section, we consider the time integration of \eqref{eq:matrix-two-field-system} by an \sstage\ implicit Runge--Kutta (RK) methods given by its matrix $\rkmatrix=(a_{ij})_{i,j=1}^\rkstages$, weights $b_i$ and nodes $c_i$, $i=1,\dots,\rkstages$, cf., \cite[Section II.1]{HaiLW06_book}.

	The two-field linear \Friedrichs system~\eqref{eq:two-field-system} fits into the general framework considered in \cite[Eq.~(2.8)]{Hochbruck_Pazur_2015} since the \Friedrichs operator \(\friedop\) is skew-adjoint on its domain \(\frieddom\). 
	Therefore, the time discretization error bound~\cite[Thm. 3.5]{Hochbruck_Pazur_2015} holds under certain regularity assumptions on the exact solution of~\eqref{eq:two-field-system}. 
	Regarding the full discretization error using a central fluxes
        discontinuous Galerkin discretization in space we refer to
       \cite[Thm~5.4 and Rem~5.7]{Hochbruck_Pazur_2015}
	Even though these results are derived for linear \Maxwells equations, it is straightforward to extend them to two-field linear \Friedrichs systems discretized in space.

To simplify the presentation, we write \eqref{eq:matrix-two-field-system} in the compact form
\begin{equation}\label{eq:semilinear_compact}
	\begin{aligned}
		\derv{t} \solvector{}
		&=
		\friedmatrix \solvector{} + \rhsvector{}, \qquad (0,T), \\
		\solvector{0} &= \solvector*{0},
	\end{aligned}
\end{equation}
where
\begin{equation*}
	\solvector{} = \begin{pmatrix}
		\solvector[first]{} \\
		\solvector[second]{}
	\end{pmatrix},
	\quad
	\rhsvector{} = \begin{pmatrix}
        \rhsvector[first]{}  \\
        \rhsvector[second]{}
    \end{pmatrix} \in \realvec{\dofs},
    \andeq
    \friedmatrix = \begin{pmatrix}
      	\zeromatrix & \friedmatrix[second] \\
      	\friedmatrix[first] & \zeromatrix
    \end{pmatrix} \in \realmat{\dofs}{\dofs},
\end{equation*}
with \(\dofs \da \dofs[first] + \dofs[second]\).
Assume that we already computed an approximation $\solvector{n}\approx \solvector*{\tn{n}}$ at time $\tn{n} = n\ts$, where $\ts>0$ denotes the step size.
Then, the implicit Runge--Kutta method applied to \eqref{eq:semilinear_compact} leads to the following coupled linear system of equations for the intermediate stages $\rkstage{i} \approx \solvector*{\tn{n} + c_i\ts}$
\begin{subequations}
\begin{equation}\label{eq:intermediate_rk}
	\rkstage{i} 
	= 
	\solvector{n} + \ts \sum_{j=1}^\rkstages a_{ij} \bigl( 
		\friedmatrix \rkstage{j} + \rkrhs{j}
	\bigr), \quad i=1,\dots,\rkstages,
\end{equation}  	
where $\rkrhs{j} = \rhsvector*{\tn{n} + c_j \ts}$.
The new approximation $\solvector{n+1} \approx \solvector*{\tn{n+1}}$ is then given explicitly by
\begin{equation}\label{eq:rk}
	\solvector{n+1}
	=
	\solvector{n}
	+
	\ts\sum_{i=1}^\rkstages b_i \bigl(\friedmatrix \rkstage{i} + \rkrhs{i} \bigr).
\end{equation}
\end{subequations}
   
\subsection{\gcs}   
We use \gcs to construct higher-order implicit RK methods. 
It is well known that these methods are algebraically stable \cite[Theorem
IV.12.9]{Hairer_Wanner_1996} and the RK matrix $\rkmatrix$ is invertible
\cite[Section~IV.14]{Hairer_Wanner_1996}. 
In addition, the error analysis for linear wave-type problems \cite[Section~3.1]{Hochbruck_Pazur_2015} makes use of the existence of a diagonal positive definite matrix $\Dmhat$ and a positive scalar $\eta>0$ such that
\begin{equation}\label{eq:coercitivity}
	\bv^\top \Dmhat\inv{\rkmatrix}\bv \geq \eta \bv^\top \Dmhat \bv, 
	\quad \text{for all} \quad \bv\in\realvec{\rkstages}.
\end{equation}  
Here, $^\top$ denotes the transpose.
For \gcs, the coercitivity condition \eqref{eq:coercitivity} is satisfied for 
$\Dmhat=\Bmhat(\Cmhat^{-1}-\idmatrix[\rkstages])$, where 
$\Bmhat\da\text{diag}(b_1,\dots,b_s)$,
$\Cmhat\da\text{diag}(c_1,\dots,c_s)$, and
$\idmatrix[\rkstages]$ is the identity matrix of size $\rkstages$, cf.\
\cite[Theorem IV.14.5]{Hairer_Wanner_1996}.

To obtain an efficient implementation for solving \eqref{eq:intermediate_rk}, we use Kronecker products \cite[Section VIII.6]{HaiLW06_book} to rewrite it as
\begin{equation}\label{eq:intermediate_rk_reduced}	
	\rkstage{} 
	= 
	\onevector[\rkstages] \otimes \solvector{n} 
	+ 
	\ts \bigl( 
		(\rkmatrix \otimes \friedmatrix) \rkstage{} 
		+ 
		(\rkmatrix \otimes \idmatrix[\dofs]) \rkrhs{}
	\bigr),
\end{equation}
where $\rkstage{} = (\rkstage{i})_{i=1}^\rkstages$, $\rkrhs{} = \bigl(\rkrhs{i}\bigr)_{i=1}^\rkstages \in \complexvec{\dofs\rkstages}$, $\idmatrix[\dofs]$ is the identity matrix of size $\dofs$, and the term $\onevector[\rkstages]$ denotes the vector in $\realvec{\rkstages}$ consisting of all ones.
Diagonalization of $\rkmatrix$ yields a nonsingular matrix $\mathbf{T}\in \complexmat{\rkstages}{\rkstages}$ containing the eigenvectors and a diagonal matrix $\lark\in \complexmat{\rkstages}{\rkstages}$ containing eigenvalues $\rkmatrixeig{i}$, such that
\begin{equation} \label{eq:eigendecomp}
	\inv{\mathbf{T}}\rkmatrix \mathbf{T} = \lark,  
	\qquad 
	\lark = \diag(\rkmatrixeig{1},\ldots,\rkmatrixeig{\rkstages}).
\end{equation}
Substituting $\rkmatrix=\mathbf{T} \lark \inv{\mathbf{T}}$ into \eqref{eq:intermediate_rk_reduced} and performing some Kronecker product operations leads to $\rkstages$ decoupled linear systems of the form
\begin{equation}\label{eq:decoupled_gauss}
    (\idmatrix[\rkstages]\otimes \idmatrix[\dofs] - \ts(\lark \otimes \friedmatrix))\rkstage*{} 
	= 
	\rkstage*{}^0 + \ts(\lark \otimes \idmatrix[\dofs])\rkrhs*{},
\end{equation} 	 
    where,  
\begin{equation}\label{eq:z}
    \rkstage*{}
	=
	(\Tinv \otimes \idmatrix[\dofs])\rkstage{},
	\quad
	\rkstage*{}^0
	=
	(\Tinv \otimes \idmatrix[\dofs])(\onevector[\rkstages] \otimes \solvector{n}),
	\quad
    \rkrhs*{}
	=
	(\Tinv \otimes \idmatrix[\dofs])\rkrhs{}.
\end{equation}

Note that $\rkmatrix$ might have complex conjugate pairs of eigenvalues. For such eigenvalues (say $\rkmatrixeig{j} = \overline{\rkmatrixeig{i}}$), the corresponding linear systems are
\begin{subequations} \label{eq:complex_conjugate_gauss}
	\begin{align}
	(\idmatrix[\dofs] - \ts \rkmatrixeig{i} \friedmatrix)\rkstage*{i} 
	&= 
	\rkstage*{i}^0
	+
	\ts\rkmatrixeig{i}\rkrhs*{i}, \label{eq:cc_gauss} \\
	(\idmatrix[\dofs] - \ts \overline{\rkmatrixeig{i}} \,\friedmatrix)\rkstage*{j} 
	&= 
	\rkstage*{j}^0
	+
	\ts \overline{\rkmatrixeig{i}}\,\rkrhs*{j}. \label{eq:cc_gauss1}
	\end{align}
\end{subequations}
In the homogeneous case, i.e., $\rhsvector{} \equiv 0$ which leads to $\rkrhs*{} \equiv 0$, the first term on the right-hand sides of  \eqref{eq:cc_gauss} and \eqref{eq:cc_gauss1} are conjugate to each other and so are the solutions.
\begin{lemma}
	If $\rhsvector{} \equiv 0$ and $\rkmatrixeig{j} = \overline{\rkmatrixeig{i}}$, then the solutions of \eqref{eq:complex_conjugate_gauss} satisfy $\rkstage*{j} = \rkstage*{i}[\overline]$.
\end{lemma}
\begin{proof}
The RK matrix $\rkmatrix$ is real and thus complex eigenvalues and eigenvectors appear in complex conjugate pairs. Hence there exists a symmetric permutation matrix $\Pmhat \in \realmat{\rkstages}{\rkstages}$ \rkstages.t., 
\begin{equation}
	\Tbar = \mathbf{T} \Pmhat, \qquad \overline{\lark} = \Pmhat \lark \Pmhat,
\end{equation}
which implies $\Tinv = \Pmhat \inv{\overline{\mathbf{T}}}$. 
We choose an arbitrary index $i \in \{0,\ldots,\rkstages\}$ corresponding to a complex eigenvalue $\rkmatrixeig{i} \not\in\mathbb{R}$ and define the index $j$ such that $\be_j = \Pmhat \be_i$. 	
By \eqref{eq:z} and $\solvector{n} \in \realvec{\dofs}$ we have 
\begin{equation}
	\rkstage*{i}[\overline]^0 
	= 
	(\be_i^\top \otimes \idmatrix[\dofs]) \rkstage*{}[\overline]^0  
	= 
	(\be_i^\top \inv{\overline{\mathbf{T}}} \onevector[\rkstages] \otimes \solvector{n}) 
	= 
	((\Pmhat\be_i)^\top \inv{\mathbf{T}} \onevector[\rkstages]) \otimes \solvector{n} 
	= 
	\rkstage*{j}^0.
\end{equation}
Conjugating  \eqref{eq:cc_gauss} proves that $\rkstage*{i}[\overline]$ solves \eqref{eq:cc_gauss1}.
\end{proof}
	
In addition to this, $\rkstage*{i}$ and $\rkrhs*{i}$ in \eqref{eq:cc_gauss} can be
further decomposed into
\begin{equation*}
	\rkstage*{i} = \begin{pmatrix}
		\rkstage*[first]{i} \\
		\rkstage*[second]{i} 
	\end{pmatrix},
	\quad
	\rkrhs*{i} = \begin{pmatrix}
		\rkrhs*[first]{i} \\	
		\rkrhs*[second]{i} 
	\end{pmatrix},
\end{equation*}
where $\rkstage*[first]{i}, \rkstage*[second]{i}$ denote unknowns corresponding to the transformed intermediate stages of $\solvector[first]{}$ and $\solvector[second]{}$ respectively. Taking the Schur complement, the linear systems in \eqref{eq:cc_gauss} can be reduced to
\begin{equation}\label{eq:schur}
	(\idmatrix[\dofs[first]] - \alpha_i\friedmatrix[second]\friedmatrix[first]) \rkstage*[first]{i} 
	=
	\rkstage*[first]{i}^0 
	+ 
	\ts \rkmatrixeig{i} (\friedmatrix[second] \rkstage*[second]{i}^0 + \rkrhs*[first]{i})
	+ 
	\alpha_i \friedmatrix[second] \rkrhs*[second]{i},
	\qquad
	\alpha_i \da \ts^2 \rkmatrixeig{i}^2 \in \complex,
\end{equation}
to compute the $\sol[first]$-component of $\rkstage*{i}$. 
After solving this linear system of dimension $\dofs[first]$, the $\sol[second]$-component of $\rkstage*{i}$ can be calculated explicitly via
\begin{equation}
  \rkstage*[second]{i}
  =
  \rkstage*[second]{i}^0 
  + 
  \ts \rkmatrixeig{i} (\friedmatrix[first] \rkstage*[first]{i} + \rkrhs*[second]{i}) .
\end{equation}
An efficient implementation of an \sstage implicit \rk using Gau{\ss} collocation points thus requires solving a linear system of the form
\begin{equation}\label{eq:linear_A}
	\systemmatrix\systemsol 
	= 
	\systemrhs
	\where
	\systemmatrix \da \idmatrix[\dofs[first]] - \alpha\friedmatrix[second]\friedmatrix[first], 
\end{equation}	
with a possibly complex parameter $\alpha$, in each time step. 
Therefore, the size of the linear system to be solved is reduced from $\rkstages(\dofs[first] + \dofs[second])$ in \eqref{eq:intermediate_rk_reduced} to $\dofs[first]$ in \eqref{eq:linear_A}.
The adjointness property \eqref{eq:full_adjoint} implies that $\friedmatrix[second]\friedmatrix[first]$ is symmetric with respect to the discrete inner product $\dgip{\cdot,\cdot}$  defined in \eqref{def:weighted_ip}. 
Hence, $\systemmatrix$ is complex symmetric, that is,
\begin{equation*}
	\dgip{\systemmatrix\systemsol, \systemsol}
	=
	\dgip{\systemsol, \overline{\systemmatrix}\systemsol},
	\qquad 
	\systemsol \in \complexvec{\dofs[first]}.
\end{equation*}   
and also normal.   
If $\alpha \in \real$, then $\systemmatrix \in \realmat{\dofs}{\dofs}$ is symmetric. 
However, for  $\alpha \not \in \real$ it follows immediately that $\systemmatrix \neq \systemmatrix^{\ast}$ with respect to $\dgip{\cdot,\cdot}$.  
Moreover, for
\begin{equation}\label{property:alpha}
	\alpha \in \complex \backslash \{z\in \real: z  < 0\},
\end{equation}    
the matrix $\systemmatrix$ is invertible. For \gcs, the coercivity condition \eqref{eq:coercitivity} guarantees that the eigenvalues of $\rkmatrix$ are not purely imaginary, and hence \eqref{property:alpha} is satisfied.

\section{Preconditioned \Krylovs}\label{sec:preconditioning}
	In this section, we aim at designing a tailored preconditioner for
	solving the sparse linear system \eqref{eq:linear_A} by a
	preconditioned \Krylov. We will prove that 
	the number of Krylov iterations 
	to achieve a certain tolerance is
	independent of the fine mesh. The overall method can be considered as
	a locally implicit scheme, because it only requires the solution of a
	small linear system as it is required for the second-order method in
	\cite{HocS16}.

	We remark that in  \Cref{sec:Krylov_results},
	we consider the Euclidean inner product and norm in $\complex^N$,
	 but using weighted inner products would be possible as well. 	
\subsection{Krylov subspace methods for complex
  symmetric matrices}\label{sec:Krylov_results}
	For a nonsingular, complex symmetric matrix $\generalMatrix=\generalMatrix^\top \in \complex^{N\times N}$ and  a given
	vector $\bff\in\complexvec{\dofs}$,  we consider the linear system
\begin{equation}\label{eq:linear}
   \generalMatrix \bx =\bff.
\end{equation}
	Given an initial guess $\bxz\in\complexvec{\dofs}$ and its initial residual vector
	$\brz=\bff-\generalMatrix\bxz$, a \Krylov\ yields an approximation of the form
\begin{equation} \label{eq:cs-qmr-x}
	\bxm=\bxz+\Wm\bym, \qquad m=1,2,\ldots,
\end{equation}
	where $\Wm\in\complex^{N\times m}$ is a basis of the $m$th Krylov subspace
\begin{equation*}
	\Kry\da \mbox{span}(\brz,\generalMatrix\brz,\dots,\generalMatrix^{m-1} \brz),
\end{equation*}
	and $\bym\in\complex^m$ is a suitable coefficient vector. The choices of  $\Wm$ and $\bym$ characterize the \Krylov,
	cf.\   \cite{Freund_Golub_Nachtigal_1992,Gutknecht_1997,Saad_2003} for more details.

	To exploit the complex symmetric structure of $\generalMatrix$,
	we suggest to use the quasi-minimal residual (QMR) algorithm for complex symmetric matrices 
	\cite[Section~3]{Freund_92}, which is based on the complex symmetric
	Lanczos process. Here, $\Wm$ satisfies
\begin{equation} \label{eq:cs-lanczos}
  	\generalMatrix  \Wm =  \Wmo \Hmt, \qquad
  	\Dmo \Hmt =  \Wmo^\top \generalMatrix \Wm, %
\end{equation}
	with a diagonal matrix
	$\Dmo = \Wmo^\top \Wmo \in \complex^{(m+1)\times (m+1)}$. The complex
	symmetry of $\generalMatrix$ implies that $\Hmt\in\complex^{(m+1) \times m}$
	is tridiagonal and the upper $m\times m$ submatrix of $\Dmo\Hmt$ is
	again complex symmetric.  $\Hmt$ has full column rank $m$ until $\Kry$
	becomes a $\generalMatrix$-invariant subspace.

	With $\beta = \norm{\brz}$, the QMR approximation is defined as
	\begin{equation} \label{eq:cs-qmr}
		\bxm = \bxz+\Wm \bym,  \qquad
		\bym = \beta \Hmt^+ \be_1,
		\qquad \Hmt^+ = (\HmtH \Hmt)^{-1} \HmtH,
	\end{equation}
	where $\be_1$ denotes the first canonical unit vector. 
	The corresponding residual can be written as
	\begin{equation*}
		\brm=\bff-\generalMatrix \bxm = \Wmo (\beta \be_1 - \Hmt \bym).
	\end{equation*}
	Since $\Hmt$ has full column rank, $\bym$ is equivalently characterized as the unique solution of the least-squares problem
	\begin{equation*}
		\bym
		= 
		\operatorname*{argmin}_{\by\in\complex^m}
		\norm{\beta\be_1-\Hmt\by}.
	\end{equation*}

	The advantage of this algorithm compared to  methods based on the
	Arnoldi process (e.g., GMRES) is that it uses three-term recurrences
	for the computation of the basis as well as for the approximation. It
	can be combined with look-ahead strategies \cite{Freund_Gutknecht_Nachtigal_1993}
	to prevent breakdowns of
	the Lanczos process, which might appear because 
	it constructs a basis which is orthogonal w.r.t.\ the
	indefinite bilinear form $\ipi{\bx,\by}=\bx^\top\by$,
	instead of the Euclidean inner product $\ip{\bx,\by}=\bx^* \by$, see
	\cite[Section~4]{Freund_92}. For the sake of presentation, we assume
	that breakdowns do not appear until a
	sufficiently accurate solution is computed, but we note that with minor 
	modifications, our analysis also holds for the (complex symmetric)
	look-ahead Lanczos method \cite{Freund_Gutknecht_Nachtigal_1993}.
	This assumption ensures that
\begin{equation}\label{eq:Dinv}
    \norm{\Dmo^{-1}} \leq \delta,
\end{equation}
	for a given (small) tolerance $\delta>0$, because otherwise,
    one would switch to the look-ahead version of the Lanczos
    process.

	In the following, $\calPm$ denotes the set of all polynomials
	over  $\complex$ of degree at most $m$.

\begin{theorem} \label{thm:cs-qmr-conv}
	Let $\generalMatrix$ be a nonsingular, complex symmetric matrix, and $\bxm$ be the
	QMR approximation \eqref{eq:cs-qmr} after $m$ steps.
	 Then the error of the QMR method satisfies
\begin{equation}\label{eq:err_General}
	\norm{\invGeneral\bff - \bxm} \le \norm{\invGeneral \Pm} \minpm\norm{\pmGeneralrz}
\end{equation}
	with a projection matrix $\Pm$ given by
\begin{equation*}
	\Pm = \idmatrix[\dofs[first]] - \Wmo  \Hmt \Hmpt  \Dmo^{-1} \Wmo^\top.
\end{equation*}
	Moreover, if  $\norm{\Wmo \be_j}=1$, $j=1,\ldots,m+1$, and \eqref{eq:Dinv} holds,
	we   have
 \begin{equation}\label{eq:normPm}
     \norm{\Pm} \leq 1 + (m+1)\delta.          
   \end{equation}
\end{theorem}
\begin{proof}
	Analogously to the proof of \cite[Theorem 2]{Hochbruck_Lubich_98} it can be seen from \eqref{eq:cs-lanczos} that $\Pm \generalMatrix\Wm=0$. Using \eqref{eq:cs-qmr} this implies
	\begin{equation*}
		\invGeneral\bff-\bxm = \invGeneral \Pm \brz   = \invGeneral \Pm \pmGeneral \brz
	\end{equation*}
	for all $p_m \in \calPm$ with $p_m(0)=1$.

	The bound on $\norm{\Pm}$ follows from \eqref{eq:Dinv} and $\norm{\Wm} \leq \matsqrt{m}$.
\end{proof}

Since $\norm{\pmGeneralrz} \leq \norm{\pmGeneral} \norm{\brz}$, it remains to bound 
\begin{equation*}
	\minpm\norm{\pmGeneral}.
\end{equation*}
This can be done by means of Faber polynomials and complex
approximation theory, cf.\ \cite{Eiermann_1989}, based on
a superset $\calS$ of the field of values of $\generalMatrix$ defined as
\begin{equation*}
	\mathcal{F}(\generalMatrix) \da
	  \{ \rhoGeneral(\bu), \bu \in \complexvec{\dofs[first]}, \bu \neq \zero\},
  \qquad \rhoGeneral(\bu) \da
         \frac{\ip{\bu,\generalMatrix\bu}}{\ip{\bu,\bu}} .
\end{equation*}

\begin{theorem} \label{thm:poly-bound-conf}
    Let $\calS\subset \complex$ be a convex and bounded superset of
    $\mathcal{F}(\generalMatrix)$ with $0 \not\in \calS$ and let $\phi$ be the
  	conformal map which
  	maps the exterior of $\calS$ onto the exterior of the unit disc with
	 $\phi(\infty)=\infty$.  Then
\begin{equation}\label{eq:poly-bound}
	\minpm\norm{\pmGeneral} \leq (1+\matsqrt{2})
    \min\left\{ \frac{3}{\abs{\phi(0)}^m},  \frac{2}{\abs{\phi(0)}^m-1}\right\}.
\end{equation}
\end{theorem}
\begin{proof}
  It was shown in \cite{CroP17} that
  \begin{equation*}
    \norm{\pmGeneral} \leq (1+\matsqrt{2})\max_{z \in \calS} \abs{p_m(z)}.
  \end{equation*}
  The statement then follows from \cite[Eq.\ (2.14)]{HocL97} and
  \cite[Theorem~2]{Eiermann_1989}.
\end{proof}

The conformal map 
$\phi$ can be determined 
numerically by using the Schwarz-Christoffel toolbox \cite{Dri05}.

\subsection{Preconditioning for locally refined grids}\label{sec:prec}
Our aim and the content of this section is the construction of a preconditioner such that the field of values of the preconditioned matrix with respect to the discrete weighted inner product $\dgip{\cdot,\cdot}$ defined in \eqref{def:weighted_ip} is independent of the fine mesh elements. %
Then by \Cref{thm:poly-bound-conf}, the same holds for the error of the
preconditioned Krylov method in this weighted inner product.

Motivated by locally implicit methods for \maxwell in
\cite{HocS16,Ver11},
we suggest to 
precondition $\systemmatrix$ from \eqref{eq:linear_A} with its dominant part, 
\begin{equation}\label{def:preconditioner}
	\systemmatrix 
	\approx 
	\preconditioner
	\da 
	\idmatrix[\dofs[first]]
	-
	\gamma\friedmatrixmodified[second]\friedmatrixmodified[first], 
\end{equation}
where $\gamma>0$ is a suitably chosen parameter.
Note that this basically boils down to replacing the \Friedrichs matrices $\friedmatrix[first], \friedmatrix[second]$ in \eqref{eq:linear_A} defined on the full mesh by the split matrices acting on the implicitly treated mesh elements, cf. \Cref{sec:problem_setting} and by replacing the possibly complex factor $\alpha$ by $\gamma>0$.
By \eqref{eq:split_adjoint} and $\gamma>0$, the preconditioner $\preconditioner$ is symmetric and positive definite with respect to $\dgip{\cdot,\cdot}$, and thus it has a symmetric and positive definite square root $\matsqrt{\preconditioner}$.
This allows us to define an equivalent preconditioned linear system 
\begin{subequations} \label{eq:linear_precon_all}
	\begin{equation}\label{eq:linear_preconditioner}
  		\systemmatrix* \systemsol* = \systemrhs*, 
	\end{equation}
	where 
	\begin{equation}\label{def:preconditioned_matrix_vector}
		\systemmatrix*
		\da
		\invsqrt{\preconditioner} \systemmatrix \invsqrt{\preconditioner},
		\quad
		\systemsol*
		\da
		\matsqrt{\preconditioner} \systemsol
		\andeq
		\systemrhs*
		\da
		\invsqrt{\preconditioner} \systemrhs .
	\end{equation}    
\end{subequations}
Since $\systemmatrix$ is complex symmetric and $\preconditioner$ is real symmetric, the preconditioned matrix $\systemmatrix*$ is again complex symmetric (with respect to $\dgip{\cdot,\cdot}$).

We now apply the QMR method for complex symmetric matrices to the
preconditioned linear system \eqref{eq:linear_precon_all} and
refer to this method as the preconditioned QMR (\PQMR)
method, cf.\
\cite[Alg.~8.1.]{Freund_Nachtigal_1994}.
Since the computation of $\matsqrt{\preconditioner}$ is usually very expensive, it is essential that it is not required for the implementation. In fact, one only has to solve linear systems with $\preconditioner$ but neither compute $\matsqrt{\preconditioner}$ nor $\invsqrt{\preconditioner}$.
Solving linear systems with $\preconditioner$ does not lead to too much
overhead costs because $\friedmatrixmodified[second]\friedmatrixmodified[first]$ only acts on the fine
elements and their direct neighbors and thus is of small
dimension compared to $\systemmatrix$.
Moreover, systems with $\preconditioner$ can also be solved iteratively, e.g., using standard preconditioning techniques for symmetric positive definite matrices.

It remains to show that the error of pQMR
can indeed be bounded
independently of the fine mesh.
Note that \Cref{thm:cs-qmr-conv,thm:poly-bound-conf}.
also hold for  $\norm{\,\cdot\,} = \dgnorm{\cdot}$, if the
Lanczos process and the field of values are defined w.r.t.\ $\ip{\cdot,\cdot} = \dgip{\cdot,\cdot}$.
Using these theorems,
it is sufficient to show that
the field of values $\fov{\systemmatrix*}$ can be bounded independent of the fine mesh.

We split the factor $\alpha$ in \eqref{eq:linear_A} into its real and imaginary part,
\begin{equation} \label{def:alpha}
	\alpha 
	\da 
	\realpha + \irm\imalpha, 
	\quad 
	\realpha, \imalpha \in \real,
\end{equation}
and define
\begin{equation}  \label{def:Grt}
	\Gtez = 1 + \zeta \opnorm*[first][second]{\friedmatrixnonstiff[first]}^2,
	\qquad \Gtiz = 1 + \zeta \opnorm*[first][second]{\friedmatrixmodified[first]}^2,
	\quad \text{for } \zeta\in\complex.                       
\end{equation}
Defining quadrilaterals
\begin{subequations} \label{eq:QR}
\begin{align}
  \Qt &= \textrm{conv}\bigl\{
        1,\;
        \Gtea,
        \tfrac{\alpha}{\gamma},\;
        \tfrac{\alpha}{\gamma} \Grtg
        \bigr\}, \label{eq:Q}\\
  \Rt &= \textrm{conv}\bigl\{
		 1,\;
		\Grtg,\;
        1+\bigl(\tfrac{\alpha}{\gamma} - 1\bigr)
        \bigl(\Grtg - \tfrac{1}{\Git}\bigr),\; 	%
        \Grtg +
        \bigl(\tfrac{\alpha}{\gamma} - 1\bigr)
        \bigl(\Grtg - \tfrac{1}{\Git}\bigr) \bigr\}     \label{eq:R}  ,\;
\end{align}
\end{subequations}
allows us to construct a superset $\calS$ of $\fov{\systemmatrix*}$ which is independent of the fine mesh.
\begin{theorem}\label{thm:fov_bound}
Let $\alpha \neq 0$ satisfy \eqref{property:alpha} and let $\systemmatrix*$ be defined in \eqref{def:preconditioned_matrix_vector} where the preconditioner $\preconditioner$ is given in \eqref{def:preconditioner} for some parameter $\gamma>0$.
Then we have $\fov{\systemmatrix*} \subset \qmrsuperset$, where
\begin{equation*}\label{eq:interval} 
	\qmrsuperset
	=
	\begin{cases}
		\Qt \cap \Rt, & \imalpha \neq 0, \\[2mm]
		[\frac{\alpha}{\gamma}, \Gtea], & \imalpha=0,
		\quad
		0 < \realpha = \alpha \leq \gamma, \\[2mm]
		[1, \frac{\alpha}{\gamma}\Grtg], & \imalpha=0,
		\quad
		0 < \gamma \leq \realpha = \alpha,
	\end{cases} 
\end{equation*}
is independent of the fine mesh and $0\not\in\qmrsuperset$.
\end{theorem}
\begin{proof}
Let $\bu \in \complexvec{\dofs[first]}, \bu\neq \zero $ and $\but\da\matsqrt{\preconditioner} \bu$. 
Then, by the symmetry of $\preconditioner$ (and thus of $\matsqrt{\preconditioner}$), the adjointness and split properties \cref{eq:full_adjoint,eq:split_adjoint,eq:split_norm}, we have
\begin{subequations} \label{eq:rhovAt}
\begin{align}
    \dgip{\but, \systemmatrix* \but} 
	&= 
	\dgip{\bu, \systemmatrix \bu}
    =  
	\dgnorm{\bu}^2 
	+ 
	(\realpha + \irm\imalpha) \bigl(
		\dgnorm{\friedmatrixnonstiff[first]\bu}^2
		+ 
		\dgnorm{\friedmatrixmodified[first]\bu}^2
	\bigr), \\
    \dgip{\but, \but} 
	&= 
	\dgip{\bu, \preconditioner\bu}
    =  
	\dgnorm{\bu}^2 
	+ 
	\gamma \dgnorm{\friedmatrixmodified[first]\bu}^2.
\end{align}
\end{subequations}

We now distinguish the two cases of $\alpha$ being real or complex.

(a) For $\imalpha \neq 0$, it is easy to see that
\begin{subequations}  \label{eq:bound1-2}
	\begin{equation}\label{eq:bound1}
	1 
	\leq 
	\re[\rhoepsAt(\but)]
	+ 
	\frac{\gamma - \realpha}{\imalpha} \im{\rhoepsAt(\but)} 
	= 
	1 
	+ 
	\frac{
		\gamma\dgnorm{\friedmatrixnonstiff[first]\bu}^2
	}{
		\dgnorm{\bu}^2 
		+ 
		\gamma \dgnorm{\friedmatrixmodified[first]\bu}^2
	}
	\leq \Grtg.
	\end{equation}
	The first inequality is obvious and the second follows from the definition of the weighted matrix norm in \eqref{eq:normweight} and $\gamma > 0$.
	In addition, we have
	\begin{equation}\label{eq:bound2}
		0 
		\leq 
		\re[\rhoepsAt(\but)]
		- 
		\frac{\realpha}{\imalpha}\im{\rhoepsAt(\but)}
		= 
		\frac{
			\dgnorm{\bu}^2
		}{
			\dgnorm{\bu}^2
			+ 
			\gamma \dgnorm{\friedmatrixmodified[first]\bu}^2
		}
		\leq 1.
	\end{equation}
\end{subequations}
A simple calculation shows that the inequalities \eqref{eq:bound1-2} are satisfied if and only if $\rhoepsAt\in\Qt$ with $\Qt$ defined in \eqref{eq:Q}.

Next we consider only the imaginary part. Using \eqref{eq:rhovAt}, and \eqref{def:Grt} we obtain
\begin{equation}\label{eq:bound3}
	0 \leq \frac{\gamma}{\imalpha} \im{\rhoepsAt(\but)}
	=
	1
	+
	\frac{
		\gamma\dgnorm{\friedmatrixnonstiff[first]\bu}^2
	}{
		\dgnorm{\bu}^2
		+
		\gamma\dgnorm{\friedmatrixmodified[first]\bu}^2
	}
	-
	\frac{
		\dgnorm{\bu}^2
	}{
		\dgnorm{\bu}^2
		+
		\gamma\dgnorm{\friedmatrixmodified[first]\bu}^2
	}
	\leq 
	\Grtg - \frac{1}{\Git}.
\end{equation}
The bounds \eqref{eq:bound1} and \eqref{eq:bound3} are satisfied if and only if $\rhoepsAt\in\Rt$ with $\Rt$ defined in \eqref{eq:R}.
Hence we proved $\fov{\systemmatrix*}\subset \Qt\cap \Rt$.

(b) For $\imalpha=0$, the matrix $\systemmatrix*\in \realmat{\dofs[first]}{\dofs[first]}$ is symmetric
	and thus $\rhoepsAt (\but) \in \real$ for all $\but \in \complexvec{\dofs[first]}$. Since 
	$\alpha = \realpha$ we have
\begin{equation}
	{\rhoepsAt(\but)} 
	=
	\frac{\alpha}{\gamma} 
	+
	\frac{
		(1 - \frac{\alpha}{\gamma})\dgnorm{\bu}^2
		+
		\alpha \dgnorm{\friedmatrixnonstiff[first]\bu}^2
	}{
		\dgnorm{\bu}^2
		+
		\gamma \dgnorm{\friedmatrixmodified[first]\bu}^2
	}. \label{eq:1}
\end{equation}
If $\alpha \geq \gamma$, \eqref{eq:1} can be bounded by
\begin{equation*}
	1 
	=
	\frac{\alpha}{\gamma} 
	+
	\frac{
		(1 - \frac{\alpha}{\gamma})\dgnorm{\bu}^2
	}{
		\dgnorm{\bu}^2
	}
	\leq 
	{\rhoepsAt(\but)}
	\leq
	\frac{\alpha}{\gamma} 
	+
	\frac{
		\alpha \dgnorm{\friedmatrixnonstiff[first]\bu}^2
	}{
		\dgnorm{\bu}^2 
	} 
	\leq 
	\frac{\alpha}{\gamma} \Grtg,
\end{equation*}	
Similarly, for $0 < \alpha \leq \gamma$, it is straightforward to see
\begin{equation*}
	\frac{\alpha}{\gamma}
	\leq 
	{\rhoepsAt(\but)}
	\leq 
	\frac{\alpha}{\gamma} 
	+
	\frac{
		(1-\frac{\alpha}{\gamma})\dgnorm{\bu}^2 
		+
		\alpha \dgnorm{\friedmatrixnonstiff[first]\bu}^2
	}{
		\dgnorm{\bu}^2
	} 
	\leq \Gtea.
\end{equation*}
Furthermore, since
\begin{equation*}
	0 \leq \Grtg - \frac{1}{\Git} \leq \Grtg, \qquad \gamma > 0,
\end{equation*}	
all quantities defining the superset
$\qmrsuperset$ are bounded independently of the implicitly
treated mesh elements and thus, $\qmrsuperset$ is
independent of $\hf$.
Finally, in all cases we have $0\notin \qmrsuperset$.
\end{proof}

Further, we point out that 
$\gamma > 0$ can be chosen freely and thus used to improve
the convergence factor. For example, a natural choice would be
\begin{equation}\label{eq:suggest}
 	\gamma = \abs{\realpha}\,\, 
	\text{if}\,\, 
	\realpha \neq 0 
	\quad
 	\text{or}
	\quad 
	\gamma = \abs{\alpha} \text{ else}.
\end{equation}
In any case, one should choose $\gamma\sim\tausq$ so that the dominating part of $\systemmatrix$ is well approximated by the preconditioner $\preconditioner$.

As a special case of \Cref{thm:fov_bound}, we obtain an inclusion set for the field of values of $\systemmatrix$ itself. 
Hence, we can state an error bound for the complex symmetric QMR method without preconditioning. 

\begin{corollary} \label{cor:fov-A} 
	For the matrix $\systemmatrix$ defined in \eqref{eq:linear_A}, Theorem~$\ref{thm:fov_bound}$ holds by substituting $\friedmatrixnonstiff[first] = \friedmatrix[first]$ and $\friedmatrixmodified[first] = 0$ in \eqref{eq:QR}.
\end{corollary}

Recall that by an inverse estimate \cite[Lemma~1.44]{DiPE12} there is a constant $c$ independent of the mesh width such that $\opnorm*[first][second]{\friedmatrix[first]} \leq c \minmeshdiam[-1]$. 
Hence, without preconditioning, the superset will scale with $\minmeshdiam[-1]$.
Applying \Cref{thm:cs-qmr-conv,thm:poly-bound-conf} to the preconditioned system \eqref{eq:linear_preconditioner} provides the following error bound:
\begin{theorem}\label{th:err_PQMR}
	Let $\qmrsol[m]$ be the QMR approximation to the solution of the preconditioned linear system \eqref{eq:linear_precon_all}.
	If \eqref{eq:Dinv} is satisfied, then there is a constant $\phi_0>1$ independent of the fine mesh such that the error of the $m$th \PQMR iterate satisfies
	\begin{equation} \label{eq:error_PQMR}
		\dgnorm{\inv{\systemmatrix*}\systemrhs* - \qmrsol[m]} 
		\leq 
		\cAt(1 + \sqrt{2}) \bigl(1+(m+1)\delta\bigr)
		\min\left\{\frac{3}{\phi_0^m}, \frac{2}{\phi_0^m-1}\right\},
	\end{equation}
	where 
	\begin{equation} \label{eq:cAt}
		\begin{cases}
			\cAt = 1, & 
			\quad\text{if}\quad 
			0 < \gamma \leq \realpha,\  \imalpha = 0 
			\quad\text{or}\quad 
			\imalpha \neq 0, \\[1mm]
			\cAt = \frac{\gamma}{\realpha}, &
			\quad\text{if}\quad 
			0 < \realpha \leq \gamma,\ \imalpha = 0.
		\end{cases} 
		\end{equation}
    \end{theorem}
\begin{proof}
Since $\systemmatrix*$ is complex symmetric, we apply \Cref{thm:cs-qmr-conv} for $\generalMatrix = \systemmatrix*$ with $\norm{\,\cdot\,} = \dgnorm{\cdot}$.
By \Cref{thm:fov_bound}, $\fov{\systemmatrix*}\subset \qmrsuperset$ is independent of the fine mesh and the same holds for the conformal map $\phi$ used in \Cref{thm:poly-bound-conf}, in particular for $\abs{\phi(0)} =: \phi_0$.
Thus, the bound \eqref{eq:error_PQMR} follows from \eqref{eq:err_General}, \eqref{eq:normPm}, and \eqref{eq:poly-bound}, if we can show
\begin{equation}\label{eq:normAtinv}
    \dgnorm{\systemmatrix*^{-1} \bw} 
	\leq 
	\cAt \dgnorm{\bw}
    \qquad\text{for all}\quad 
	\bw \in \complexvec{\dofs[first]}.
\end{equation}
We choose an arbitrary $\bw \in \complexvec{\dofs[first]}$, $\bw \neq \zero$  and define $\bu = \inv{\systemmatrix*}\bw$. Then, \Cref{thm:fov_bound} with $\cAt$ defined in \eqref{eq:cAt} and the Cauchy-Schwarz inequality yield
\begin{equation}\label{eq:invAt}
	\frac{1}{\cAt} 
	\leq 
	\re[\rho_{\systemmatrix*}(\bu)]
	=
	\re\frac{
		\dgip{\bw,\inv{\systemmatrix*}\bw}
	}{
		\dgnorm{\inv{\systemmatrix*}\bw}^2
	}
	\leq	
	\frac{
		\dgnorm{\bw}\dgnorm{\inv{\systemmatrix*}\bw}
	}{
		\dgnorm{\inv{\systemmatrix*}\bw}^2
	}
	=
	\frac{
		\dgnorm{\bw}
	}{
		\dgnorm{\inv{\systemmatrix*}\bw}
	}.
\end{equation}
This proves \eqref{eq:normAtinv}.
\end{proof}
As an immediate consequence of \Cref{th:err_PQMR} we see that the error of the \PQMR method is bounded independently of the fine mesh, since $\qmrsuperset$ only depends on the coarse mesh. 
In particular, the number of iterations is uniformly bounded with respect to further refinement of the fine part of the mesh.

\section{Numerical experiments}\label{sec:numerical_exp}
For all experiments, we consider the transverse-electric (TE) polarization of linear \maxwell in a homogeneous medium $\Omega=(0,1)^2$ with $\mu=\epsilon=1$, i.e.,
\begin{subequations}\label{eq:TE-system}
    \begin{align}
        \derv{t} \exE_x &= \partial_y \exH_z - \exJ_x, && \domain\times(0,T), \\
        \derv{t} \exE_y &= -\partial_x \exH_z - \exJ_y, && \domain\times(0,T), \\
        \derv{t} \exH_z &= \partial_y \exE_x - \partial_x \exE_y, && \domain\times(0,T), \\
        \exE(0) &= \exE^0, \quad \exH(0) = \exH^0, && \domain,
        \\
        \exE \times \normal{} &= 0, && \partial\domain\times(0,T) .
    \end{align}
\end{subequations}
The codes to reproduce our results are available at
\begin{center}
    \url{https://github.com/MalikScheifinger/HigherOrderLTI}
\end{center}
The software is based on the FEM library \href{https://www.dealii.org}{\texttt{deal.II}}~\cite{dealII25} version 9.7 and the Maxwell toolbox \href{https://gitlab.kit.edu/kit/ianm/ag-numerik/projects/dg-maxwell/timaxdg}{\texttt{TiMaxdG}}~\cite{TiMaxdG}.

\subsection{Illustration of Theorems~\ref{thm:fov_bound} and~\ref{th:err_PQMR}}

We consider a dG discretization with polynomials of degree two and fix
\begin{equation}\label{eq:choice_alpha_gamma}
	\alpha = \Bigl(\tfrac{1}{24} + \irm\tfrac{\sqrt{3}}{24}\Bigr)\tausq,
	\qquad
	\gamma = \realpha = \tfrac{1}{24} \tausq,
\end{equation}
for the system matrix $\systemmatrix$ in \eqref{eq:linear_A} and the preconditioner $\preconditioner$ in \eqref{def:preconditioner}, respectively.
This choice follows from $\alpha_i = \tausq\rkmatrixeig{i}^2$ in \eqref{eq:schur}, with $\rkmatrixeig{i}$ one of the two complex conjugate eigenvalues of the RK matrix $\rkmatrix$ of the two-stage implicit Gau{\ss}--Legendre method. The parameter $\gamma = \realpha$ is the choice suggested in \eqref{eq:suggest}. The spatial domain is discretized by a uniform coarse mesh which is then locally refined at the center, yielding a coarse mesh width \(\meshdiamCoarse = 0.117851\) and fine mesh widths \(\meshdiamFine[(2)] = 0.0294628\), \(\meshdiamFine[(3)] = 0.0147314\), and \(\meshdiamFine[(4)] = 0.0073657\), where the superscript denotes the local refinement level. Local refinement level three is illustrated in \Cref{fig:mesh}.

\begin{figure}[!htb]
    \centering
    \input{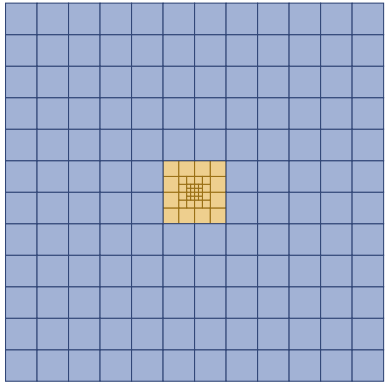}
    \caption{Mesh with local refinement level three at the center of the domain.}
    \label{fig:mesh}
\end{figure}

\Cref{fig:fov} shows the spectrum of \(\systemmatrix\) from \eqref{eq:linear_A}, the field of values of the preconditioned matrix \(\systemmatrix*\) from \eqref{def:preconditioned_matrix_vector}, and the quadrilateral bounds \(\Qt\) and \(\Rt\) from \eqref{eq:Q} and \eqref{eq:R}, for time step size \(\ts = 0.05\) and local refinement level three.
By the inverse estimate used in the proof of \Cref{cor:fov-A}, the eigenvalues of \(\systemmatrix\) scale like \(\minmeshdiam[-1]\). The spectrum of the unpreconditioned system is therefore governed by the fine elements and spreads out as the mesh is refined.

Next, we apply the QMR method to the system \eqref{eq:linear_A}, with and without the preconditioner~\eqref{def:preconditioner}. 
To obtain an algebraic test problem with known solution, we choose a random vector as the exact solution \(\systemsol\) of the linear system \eqref{eq:linear_A} and define the right-hand side \(\systemrhs\) by applying the matrix \(\systemmatrix\) to it.
This allows us to compare the relative \(\Ltwo\)-error of the QMR iterates directly. \Cref{fig:QMRvsPQMR_tau0p05} shows the results.
For a fixed tolerance, the number of preconditioned QMR iterations is almost independent of the local refinement level, whereas the number of QMR iterations grows with it, as predicted by \Cref{th:err_PQMR} together with \Cref{cor:fov-A}.

The bound of \Cref{thm:fov_bound} is fine-mesh-independent but not tight when \(\ts\) is small relative to the fine mesh width, where the linear system is only mildly stiff.
In this regime the fine-mesh dependence of the superset enters solely through the term \(\Grtg - 1/\Git\) of \(\Rt\) in \eqref{eq:R}: the factor \(1/\Git\) with \(\Git\) from \eqref{def:Grt} is not negligible at on the coarsest level of refinement and vanishes only as the fine mesh is refined, so \(\Rt\) grows toward its level-independent limit.
\Cref{fig:QMRvsPQMR_tau0p01} shows this effect for \(\tau=0.01\): preconditioned QMR still improves the convergence, and the iteration count increases with the local refinement level while remaining bounded by the fine-mesh-independent superset $\calS$ of \Cref{thm:fov_bound}.
\Cref{fig:fov_tau0p01} illustrates the same dependence geometrically, where \(\Qt\) is identical for all levels and only \(\Rt\) grows.
Such small time steps are not used in practice for the coarse spatial discretizations considered here: once the time discretization error drops below the spatial discretization error, decreasing \(\ts\) further does not reduce the total error.

\begin{figure}[!htb]
    \centering
	\input{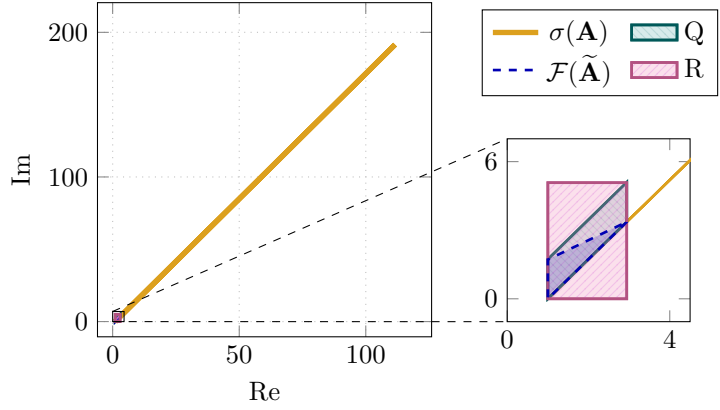}
    \caption{Spectrum \(\sigma(\systemmatrix)\) (orange), field of values \(\mathcal{F}(\systemmatrix*)\) of the preconditioned matrix (blue, dashed), and quadrilateral bounds \(\Qt\) (teal) and \(\Rt\) (magenta) with \(\alpha\) chosen as in \eqref{eq:choice_alpha_gamma} and \(\tau = 0.05\), local refinement level three. The inset shows a zoom near the origin.}
    \label{fig:fov}
\end{figure}

\begin{figure}[!htb]
    \centering
	\input{tikz/QMRvsPQMR_tau0p05.tex}
    \caption{Relative \(L^2\) error versus iteration count for QMR (solid) and precond. QMR (dashed) applied to \eqref{eq:linear_A} with \(\alpha\) chosen as in \eqref{eq:choice_alpha_gamma} and \(\tau = 0.05\). Colors denote local refinement levels: 2 (green), 3 (blue), 4 (red).}
    \label{fig:QMRvsPQMR_tau0p05}
\end{figure}

\begin{figure}[!htb]
    \centering
	\input{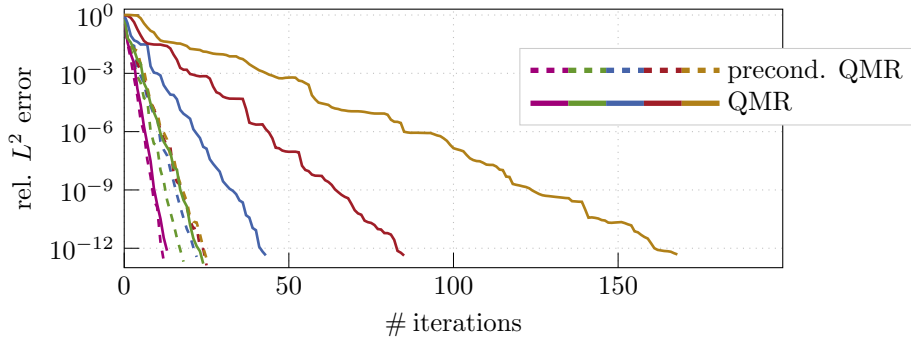}
    \caption{Relative \(L^2\) error versus iteration count for QMR (solid) and precond. QMR (dashed) applied to \eqref{eq:linear_A} with \(\alpha\) chosen as in \eqref{eq:choice_alpha_gamma} and \(\tau = 0.01\). Colors denote local refinement levels: 1 (purple), 2 (green), 3 (blue), 4 (red), 5 (orange).}
    \label{fig:QMRvsPQMR_tau0p01}
\end{figure}

\begin{figure}[!htb]
    \centering
	\input{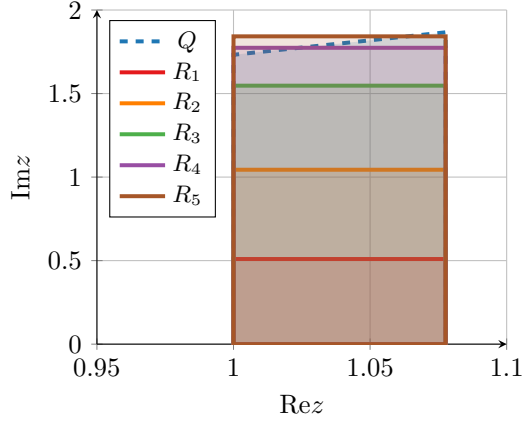}
    \caption{Quadrilateral bounds \(\Qt\) (blue, dashed, identical for all levels) and \(\Rt_k\) (solid) for local refinement levels \(k=1,\ldots,5\) (red, orange, green, purple, brown) for preconditioned system~\eqref{def:preconditioned_matrix_vector} with \eqref{eq:choice_alpha_gamma} and \(\tau = 0.01\).}
    \label{fig:fov_tau0p01}
\end{figure}

\subsection{Runtime comparison}

In this section, we compare the methods with respect to runtime and time integration accuracy.
We use the same TE model problem, but now choose a manufactured solution which is a polynomial of degree seven in each variable.
The spatial discretization uses dG polynomials of degree four.
Thus, the exact solution is not represented exactly by the spatial discretization, and the measured error contains the full discretization error. To further reduce the effect of superconvergence, we randomly perturb the mesh.
The locally refined mesh has
\(\meshdiamFine = 0.00184142\), \(\meshdiamCoarse = 0.0147314\), and \(\dofs = 693\,900\).
The \(\Ltwo\)-error is evaluated at the final time \(T=1\).

\Cref{fig:timings} compares Gau{\ss} RK methods with two (purple) and three (orange) stages, with (dashed) and without (solid) preconditioner \eqref{def:preconditioner}, the locally implicit method~\cite{HocS16} (blue), a local time-stepping method~\cite{Hochbruck_Scheifinger_2026} (green), and the popular low-storage Runge--Kutta method LSRK(12,4) \cite{Busch_Diehl_Niegemann_2021} (red). The local time-stepping method uses a Chebyshev-type iteration on the local system instead of solving a linear system like the locally implicit method, thus resulting in a fully explicit method. The degree of the Chebyshev polynomial must be chosen large enough to obtain a CFL condition independent of the stiff part of the system.
For the two- and three-stage Gau{\ss} RK methods, the preconditioner significantly reduces the runtime at comparable errors.
The savings come from the reduced number of QMR iterations for the preconditioned system \eqref{eq:linear_preconditioner}, while the application of the preconditioner \eqref{def:preconditioner} only requires solves on the fine part and its neighboring coarse elements.
The locally implicit and local time-stepping methods are efficient for admissible \timestepsizes, but are still restricted by a CFL condition on the coarse mesh and become unstable for larger \timestepsizes. Moreover, LSRK(12,4) is only stable for \timestepsizes so small that the spatial error dominates. Though it is faster than the unpreconditioned Gau{\ss} RK methods, it is slower than the preconditioned variants at nearly the same full discretization accuracy.
The Gau{\ss} RK methods are unconditionally stable and therefore remain applicable for all \timestepsizes.

\begin{figure}[!htb]
    \centering
    \input{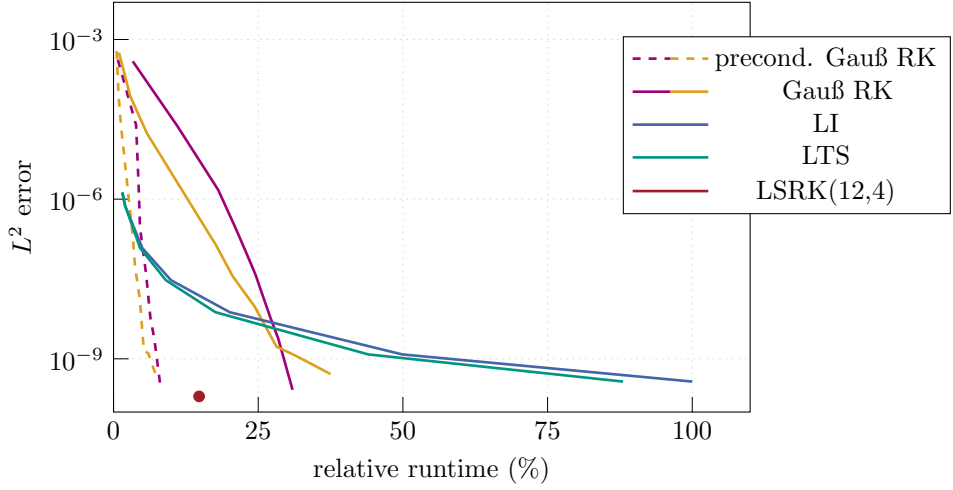}
    \caption{\(L^2\) error at final time \(T=1\) versus relative runtime for Gau{\ss} RK with two stages (purple) and three stages (orange), with preconditioner~\eqref{def:preconditioner} (dashed) and without (solid); the locally implicit method (blue); the local time-stepping method (green); and LSRK(12,4) (red marker).}
    \label{fig:timings}
\end{figure}

\section{Conclusion}\label{sec:conclusion}
In this paper, we proposed and analyzed computationally efficient implicit higher-order time integration methods for solving linear \Friedrichs systems on locally refined spatial grids which consist of a small number of fine and a large number of coarse mesh elements.
This is achieved by constructing a preconditioned Krylov subspace method to solve the linear systems arising in each time step of the implicit scheme. 
Our main result shows that the number of Krylov steps to achieve the desired accuracy can be bounded independently of the fine mesh.

Although we focused on linear \Friedrichs systems, our ideas carry over to nonlinear Friedrichs problems like nonlinear \Maxwells equations, where linear systems of the same type appear in each iteration of a (simplified) Newton method. 
Moreover, instead of Gau{\ss} collocation methods other implicit time integration schemes might be employed and the preconditioner can also be combined with rational Krylov subspace methods.

\section*{Acknowledgments} 
The authors thank Pratik Kumbhar for helpful discussions and for providing preliminary numerical examples in Matlab. 

\bibliographystyle{siamplain}
\bibliography{references}

@article {Hochbruck_Lubich_98,
	AUTHOR 		= {Hochbruck, Marlis and Lubich, Christian},
	TITLE 		= {Error analysis of {K}rylov methods in a nutshell},
	JOURNAL 	= {SIAM J. Sci. Comput.},
	FJOURNAL 	= {SIAM Journal on Scientific Computing},
	VOLUME 		= {19},
	YEAR 		= {1998},
	NUMBER 		= {2},
	PAGES 		= {695--701},
	ISSN 		= {1064-8275},
	MRCLASS 	= {65F10},
	MRNUMBER 	= {1618868},
	DOI 		= {10.1137/S1064827595290450},
}

@Article{dealII25,
  author  = {Daniel Arndt and Wolfgang Bangerth and Maximilian Bergbauer and Bruno Blais and Marc Fehling and Rene Gassm\"{o}ller
             and Timo Heister and Luca Heltai and Martin Kronbichler and Matthias Maier and Peter Munch and Sam Scheuerman
             and Bruno Turcksin and Siarhei Uzunbajakau and David Wells and Micha{\l} Wichrowski},
  title   = {The deal.II library, Version 9.7},
  journal = {Journal of Numerical Mathematics},
  year    = 2025,
  volume  = 33,
  number  = 4,
  pages   = {403--415},
  doi     = {10.1515/jnma-2025-0115}
}

@misc{TiMaxdG,
  author = {Constantin Carle and Julian D{\"o}rner and Jonas K{\"o}hler and Jan Leibold and Bernhard Maier},
  title = {{TiMaxdG}},
  url = {https://gitlab.kit.edu/kit/ianm/ag-numerik/projects/dg-maxwell/timaxdg}
}

@article {Hochbruck_Scheifinger_2026,
    AUTHOR = {Hochbruck, Marlis and Scheifinger, Malik},
     TITLE = {Local {T}ime {I}ntegration for {F}riedrichs' {S}ystems},
   JOURNAL = {SIAM J. Numer. Anal.},
  FJOURNAL = {SIAM Journal on Numerical Analysis},
    VOLUME = {64},
      YEAR = {2026},
    NUMBER = {2},
     PAGES = {370--390},
      ISSN = {0036-1429,1095-7170},
   MRCLASS = {65M12 (65M15 65M60)},
  MRNUMBER = {5041821},
       DOI = {10.1137/25M1735627},
}

@article {Busch_Diehl_Niegemann_2021,
    AUTHOR = {Niegemann, Jens and Diehl, Richard and Busch, Kurt},
     TITLE = {Efficient low-storage {R}unge-{K}utta schemes with optimized
              stability regions},
   JOURNAL = {J. Comput. Phys.},
  FJOURNAL = {Journal of Computational Physics},
    VOLUME = {231},
      YEAR = {2012},
    NUMBER = {2},
     PAGES = {364--372},
      ISSN = {0021-9991,1090-2716},
   MRCLASS = {65L06 (65M20)},
  MRNUMBER = {2872080},
       DOI = {10.1016/j.jcp.2011.09.003}
}

@Article{Yee_66,
	Title                    = {Numerical solution of initial boundary value problems involving {M}axwell's equations in isotropic media},
	Author                   = {Kane Yee},
	Journal                  = {IEEE Transactions on Antennas and Propagation},
	Year                     = {1966},
	Pages                    = {302--307},
	Volume                   = {14},
	Number                   = {3},
	URL                      = {https://ieeexplore.ieee.org/document/1138693},
}

@Article{Freund_92,
	AUTHOR		 = {Freund, Roland W.},
	TITLE 		 = {Conjugate gradient-type methods for linear systems with
					complex symmetric coefficient matrices},
	JOURNAL 	 = {SIAM J. Sci. Statist. Comput.},
	FJOURNAL 	 = {Society for Industrial and Applied Mathematics. Journal on
					Scientific and Statistical Computing},
	VOLUME 	   	 = {13},
	YEAR 		 = {1992},
	NUMBER 		 = {1},
	PAGES 		 = {425--448},
	ISSN 		 = {0196-5204},
	MRCLASS 	 = {65F10 (65F20 65F50 65N22)},
	MRNUMBER 	 = {1145195},
	DOI	 	     = {10.1137/0913023},
	URL 		 = {https://doi.org/10.1137/0913023},
}

@Article{Freund_Nachtigal_1991,
  AUTHOR 	= {Freund, Roland W. and Nachtigal, No\"{e}l M.},
  TITLE 	= {Q{MR}: a quasi-minimal residual method for non-{H}ermitian
  				linear systems},
  JOURNAL 	= {Numer. Math.},
  FJOURNAL 	= {Numerische Mathematik},
  VOLUME 	= {60},
  YEAR 		= {1991},
  NUMBER 	= {3},
  PAGES 	= {315--339},
  ISSN 		= {0029-599X},
  MRCLASS 	= {65F10},
  MRNUMBER 	= {1137197},
  DOI 		= {10.1007/BF01385726},
  URL 		= {https://doi.org/10.1007/BF01385726},
}

@Article{HocS16,
  AUTHOR 	= {Hochbruck, Marlis and Sturm, Andreas},
  TITLE 	= {Error analysis of a second-order locally implicit method for
			  linear {M}axwell's equations},
  JOURNAL 	= {SIAM J. Numer. Anal.},
  FJOURNAL 	= {SIAM Journal on Numerical Analysis},
  VOLUME 	= {54},
  YEAR 		= {2016},
  NUMBER 	= {5},
  PAGES 	= {3167--3191},
  ISSN 		= {0036-1429},
  MRCLASS 	= {65M60 (65M12 65M15 78A30)},
  MRNUMBER 	= {3565552},
  MRREVIEWER = {Mario Annunziato},
  DOI 		= {10.1137/15M1038037},
  URL 		= {https://doi.org/10.1137/15M1038037},
}

@Article{HocS19,
	AUTHOR 		= {Hochbruck, Marlis and Sturm, Andreas},
	TITLE 		= {Upwind discontinuous {G}alerkin space discretization and
					locally implicit time integration for linear {M}axwell's equations},
	JOURNAL 	= {Math. Comp.},
	FJOURNAL 	= {Mathematics of Computation},
	VOLUME 		= {88},
	YEAR 		= {2019},
	NUMBER 		= {317},
	PAGES 		= {1121--1153},
	ISSN 		= {0025-5718},
	MRCLASS 	= {65M60 (65M12 65M15 78A25)},
	MRNUMBER 	= {3904141},
	MRREVIEWER 	= {Roberto C. Cabrales},
	DOI 		= {10.1090/mcom/3365},
	URL 		= {https://doi.org/10.1090/mcom/3365},
}

@Article{Eiermann_1989,
	AUTHOR = {Eiermann, Michael},
	TITLE = {On semiiterative methods generated by {F}aber polynomials},
	JOURNAL = {Numer. Math.},
	FJOURNAL = {Numerische Mathematik},
	VOLUME = {56},
	YEAR = {1989},
	NUMBER = {2-3},
	PAGES = {139--156},
	ISSN = {0029-599X},
	MRCLASS = {65F10},
	MRNUMBER = {1018298},
	MRREVIEWER = {Jin Xi Zhao},
	DOI = {10.1007/BF01409782},
	URL = {https://doi.org/10.1007/BF01409782},
}

@Article{Freund_Gutknecht_Nachtigal_1993,
	AUTHOR = {Freund, Roland W. and Gutknecht, Martin H. and Nachtigal, No\"{e}l
	M.},
	TITLE = {An implementation of the look-ahead {L}anczos algorithm for
	non-{H}ermitian matrices},
	JOURNAL = {SIAM J. Sci. Comput.},
	FJOURNAL = {SIAM Journal on Scientific Computing},
	VOLUME = {14},
	YEAR = {1993},
	NUMBER = {1},
	PAGES = {137--158},
	ISSN = {1064-8275},
	MRCLASS = {65F15},
	MRNUMBER = {1201315},
	DOI = {10.1137/0914009},
	URL = {https://doi.org/10.1137/0914009},
}

@article {Hochbruck_Pazur_2015,
	AUTHOR = {Hochbruck, Marlis and Pa\v{z}ur, Tomislav},
	TITLE = {Implicit {R}unge-{K}utta methods and discontinuous {G}alerkin
	discretizations for linear {M}axwell's equations},
	JOURNAL = {SIAM J. Numer. Anal.},
	FJOURNAL = {SIAM Journal on Numerical Analysis},
	VOLUME = {53},
	YEAR = {2015},
	NUMBER = {1},
	PAGES = {485--507},
	ISSN = {0036-1429},
	MRCLASS = {65M60 (35Q61 65J08 65M12 65M15 78M10)},
	MRNUMBER = {3313827},
	MRREVIEWER = {Rados\l aw Szmytkowski},
	DOI = {10.1137/130944114},
	URL = {https://doi.org/10.1137/130944114},
}

@article {DesLM13,
	AUTHOR = {Descombes, St\'{e}phane and Lanteri, St\'{e}phane and Moya, Ludovic},
	TITLE = {Locally implicit time integration strategies in a
	discontinuous {G}alerkin method for {M}axwell's equations},
	JOURNAL = {J. Sci. Comput.},
	FJOURNAL = {Journal of Scientific Computing},
	VOLUME = {56},
	YEAR = {2013},
	NUMBER = {1},
	PAGES = {190--218},
	ISSN = {0885-7474},
	MRCLASS = {65M60 (65M22 78M10)},
	MRNUMBER = {3049948},
	MRREVIEWER = {H. P. Dikshit},
	DOI = {10.1007/s10915-012-9669-5},
	URL = {https://doi.org/10.1007/s10915-012-9669-5},
}

@Article{DesLM17,
	AUTHOR = {Descombes, St\'{e}phane and Lanteri, St\'{e}phane and Moya, Ludovic},
	TITLE = {Temporal convergence analysis of a locally implicit
	discontinuous {G}alerkin time domain method for
	electromagnetic wave propagation in dispersive media},
	JOURNAL = {J. Comput. Appl. Math.},
	FJOURNAL = {Journal of Computational and Applied Mathematics},
	VOLUME = {316},
	YEAR = {2017},
	PAGES = {122--132},
	ISSN = {0377-0427},
	MRCLASS = {65M60 (65M12 78A25 94A08)},
	MRNUMBER = {3588733},
	MRREVIEWER = {Georgios D. Akrivis},
	DOI = {10.1016/j.cam.2016.09.038},
	URL = {https://doi.org/10.1016/j.cam.2016.09.038},
}

@Article{Chabassier_Imperiale_2016,
	AUTHOR = {Chabassier, J. and Imperiale, S.},
	TITLE = {Fourth-order energy-preserving locally implicit time
	discretization for linear wave equations},
	JOURNAL = {Internat. J. Numer. Methods Engrg.},
	FJOURNAL = {International Journal for Numerical Methods in Engineering},
	VOLUME = {106},
	YEAR = {2016},
	NUMBER = {8},
	PAGES = {593--622},
	ISSN = {0029-5981},
	MRCLASS = {65M60 (65M12)},
	MRNUMBER = {3491710},
	MRREVIEWER = {H. P. Dikshit},
	DOI = {10.1002/nme.5130},
	URL = {https://doi.org/10.1002/nme.5130},
}

@Article{Ver11,
	AUTHOR = {Verwer, J. G.},
	TITLE = {Component splitting for semi-discrete {M}axwell equations},
	JOURNAL = {BIT},
	FJOURNAL = {BIT. Numerical Mathematics},
	VOLUME = {51},
	YEAR = {2011},
	NUMBER = {2},
	PAGES = {427--445},
	ISSN = {0006-3835},
	MRCLASS = {65M20 (65L05 65L20 65M06 78M10 78M20)},
	MRNUMBER = {2806538},
	MRREVIEWER = {M. K. Kadalbajoo},
	DOI = {10.1007/s10543-010-0296-y},
	URL = {https://doi.org/10.1007/s10543-010-0296-y},
}

@Article{Pip06,
	AUTHOR = {Piperno, Serge},
	TITLE = {Symplectic local time-stepping in non-dissipative {DGTD}
	methods applied to wave propagation problems},
	JOURNAL = {M2AN Math. Model. Numer. Anal.},
	FJOURNAL = {M2AN. Mathematical Modelling and Numerical Analysis},
	VOLUME = {40},
	YEAR = {2006},
	NUMBER = {5},
	PAGES = {815--841 (2007)},
	ISSN = {0764-583X},
	MRCLASS = {65M70 (37M15 65M12 78M10)},
	MRNUMBER = {2293248},
	DOI = {10.1051/m2an:2006035},
	URL = {https://doi.org/10.1051/m2an:2006035},
}

@Article{GroM10,
	AUTHOR = {Grote, Marcus J. and Mitkova, Teodora},
	TITLE = {Explicit local time-stepping methods for {M}axwell's
	equations},
	JOURNAL = {J. Comput. Appl. Math.},
	FJOURNAL = {Journal of Computational and Applied Mathematics},
	VOLUME = {234},
	YEAR = {2010},
	NUMBER = {12},
	PAGES = {3283--3302},
	ISSN = {0377-0427},
	MRCLASS = {65M60 (78M10)},
	MRNUMBER = {2665386},
	MRREVIEWER = {Vrushali A. Bokil},
	DOI = {10.1016/j.cam.2010.04.028},
	URL = {https://doi.org/10.1016/j.cam.2010.04.028},
}

@Article{GroMM15,
	AUTHOR = {Grote, Marcus J. and Mehlin, Michaela and Mitkova, Teodora},
	TITLE = {Runge-{K}utta-based explicit local time-stepping methods for
	wave propagation},
	JOURNAL = {SIAM J. Sci. Comput.},
	FJOURNAL = {SIAM Journal on Scientific Computing},
	VOLUME = {37},
	YEAR = {2015},
	NUMBER = {2},
	PAGES = {A747--A775},
	ISSN = {1064-8275},
	MRCLASS = {65M60 (65L06 65M50)},
	MRNUMBER = {3324978},
	MRREVIEWER = {Istv\'{a}n Farag\'{o}},
	DOI = {10.1137/140958293},
	URL = {https://doi.org/10.1137/140958293},
}

@Article{DiaG09,
	AUTHOR = {Diaz, Julien and Grote, Marcus J.},
	TITLE = {Energy conserving explicit local time stepping for
	second-order wave equations},
	JOURNAL = {SIAM J. Sci. Comput.},
	FJOURNAL = {SIAM Journal on Scientific Computing},
	VOLUME = {31},
	YEAR = {2009},
	NUMBER = {3},
	PAGES = {1985--2014},
	ISSN = {1064-8275},
	MRCLASS = {65M60 (65M12)},
	MRNUMBER = {2516141},
	MRREVIEWER = {Juhani Pitk\"{a}ranta},
	DOI = {10.1137/070709414},
	URL = {https://doi.org/10.1137/070709414},
}

@Book{HaiLW06_book,
	AUTHOR = {Hairer, Ernst and Lubich, Christian and Wanner, Gerhard},
	TITLE = {Geometric numerical integration},
	SERIES = {Springer Series in Computational Mathematics},
	VOLUME = {31},
	EDITION = {Second},
	NOTE = {Structure-preserving algorithms for ordinary differential
	equations},
	PUBLISHER = {Springer-Verlag, Berlin},
	YEAR = {2006},
	PAGES = {xviii+644},
	ISBN = {3-540-30663-3; 978-3-540-30663-4},
	MRCLASS = {65-02 (37M15 65Lxx 65P10 70-08)},
	MRNUMBER = {2221614},
}

@article{HocK22,
    AUTHOR = {Hochbruck, Marlis and K\"{o}hler, Jonas},
     TITLE = {Error analysis of a fully discrete discontinuous {G}alerkin
              alternating direction implicit discretization of a class of
              linear wave-type problems},
   JOURNAL = {Numer. Math.},
  FJOURNAL = {Numerische Mathematik},
    VOLUME = {150},
      YEAR = {2022},
    NUMBER = {3},
     PAGES = {893--927},
      ISSN = {0029-599X},
   MRCLASS = {65M60 (65M12 65M15 65M22)},
  MRNUMBER = {4394004},
MRREVIEWER = {Murat Uzunca},
       DOI = {10.1007/s00211-021-01262-z}
}

@book{MFO2023,
  title = {Wave {{Phenomena}}: {{Mathematical Analysis}} and {{Numerical Approximation}}},
  shorttitle = {Wave {{Phenomena}}},
  author = {D{\"o}rfler, Willy and Hochbruck, Marlis and K{\"o}hler, Jonas and Rieder, Andreas and Schnaubelt, Roland and Wieners, Christian},
  year = {2023},
  series = {Oberwolfach {{Seminars}}},
  volume = {49},
  publisher = {{Springer International Publishing}},
  address = {{Cham}},
  url = {https://link.springer.com/10.1007/978-3-031-05793-9},
  isbn = {978-3-031-05792-2 978-3-031-05793-9}
}

@Book{Hairer_Wanner_1996,
	AUTHOR = {Hairer, E. and Wanner, G.},
	TITLE = {Solving ordinary differential equations. {II}},
	SERIES = {Springer Series in Computational Mathematics},
	VOLUME = {14},
	EDITION = {Second},
	NOTE = {Stiff and differential-algebraic problems},
	PUBLISHER = {Springer-Verlag, Berlin},
	YEAR = {1996},
	PAGES = {xvi+614},
	ISBN = {3-540-60452-9},
	MRCLASS = {65-02 (34A09 34A45 65-01 65Lxx)},
	MRNUMBER = {1439506},
	DOI = {10.1007/978-3-642-05221-7},
	URL = {https://doi.org/10.1007/978-3-642-05221-7},
}

@Book{DiPE12,
	AUTHOR = {Di Pietro, Daniele Antonio and Ern, Alexandre},
	TITLE = {Mathematical aspects of discontinuous {G}alerkin methods},
	SERIES = {Math\'{e}matiques \& Applications (Berlin) [Mathematics \&
	Applications]},
	VOLUME = {69},
	PUBLISHER = {Springer, Heidelberg},
	YEAR = {2012},
	PAGES = {xviii+384},
	ISBN = {978-3-642-22979-4},
	MRCLASS = {65-02 (35A35 35F15 35J25 35Q35 65M60 65N30)},
	MRNUMBER = {2882148},
	MRREVIEWER = {R\'{e}mi Vaillancourt},
	DOI = {10.1007/978-3-642-22980-0},
	URL = {https://doi.org/10.1007/978-3-642-22980-0},
}

@Book{Saad_2003,
  AUTHOR = {Saad, Yousef},
  TITLE = {Iterative methods for sparse linear systems},
  EDITION = {Second},
  PUBLISHER = {Society for Industrial and Applied Mathematics, Philadelphia,
  PA},
  YEAR = {2003},
  PAGES = {xviii+528},
  ISBN = {0-89871-534-2},
  MRCLASS = {65-01 (65F10 65F50)},
  MRNUMBER = {1990645},
  MRREVIEWER = {Arnold Reusken},
  DOI = {10.1137/1.9780898718003},
  URL = {https://doi.org/10.1137/1.9780898718003},
}

@article {Dri05,
    AUTHOR = {Driscoll, Tobin A.},
     TITLE = {Algorithm 843: improvements to the {S}chwarz-{C}hristoffel
              toolbox for {MATLAB}},
   JOURNAL = {ACM Trans. Math. Software},
  FJOURNAL = {Association for Computing Machinery. Transactions on
              Mathematical Software},
    VOLUME = {31},
      YEAR = {2005},
    NUMBER = {2},
     PAGES = {239--251},
      ISSN = {0098-3500},
   MRCLASS = {30-04 (30C30 65E05 65Y15)},
  MRNUMBER = {2266791},
       DOI = {10.1145/1067967.1067971},
       URL = {https://doi.org/10.1145/1067967.1067971},
}

@InProceedings{HocK20,
	author          =  "Hochbruck, Marlis and K{\"o}hler, Jonas",
	editor          =  "D{\"o}rfler, Willy	and Hochbruck, Marlis	and Hundertmark, Dirk
						and Reichel, Wolfgang and Rieder, Andreas and Schnaubelt, Roland and Sch{\"o}rkhuber, Birgit",
	title           = "Error {A}nalysis of {D}iscontinuous {G}alerkin {D}iscretizations of a {C}lass of {L}inear
						 {W}ave-type {P}roblems",
	booktitle		= "Mathematics of Wave Phenomena",
	year			= "2020",
	publisher		= "Springer International Publishing",
	address			= "Cham",
	pages			= "197--218",
    url				= "https://link.springer.com/chapter/10.1007%2F978-3-030-47174-3_12",
}

@article{Freund_Nachtigal_1994,
    AUTHOR = {Freund, Roland W. and Nachtigal, No\"{e}l M.},
     TITLE = {An implementation of the {QMR} method based on coupled
              two-term recurrences},
      NOTE = {Iterative methods in numerical linear algebra (Copper Mountain
              Resort, CO, 1992)},
   JOURNAL = {SIAM J. Sci. Comput.},
  FJOURNAL = {SIAM Journal on Scientific Computing},
    VOLUME = {15},
      YEAR = {1994},
    NUMBER = {2},
     PAGES = {313--337},
      ISSN = {1064-8275},
   MRCLASS = {65F10},
  MRNUMBER = {1261456},
MRREVIEWER = {David F. Griffiths},
       DOI = {10.1137/0915022},
       URL = {https://doi.org/10.1137/0915022},
}

@incollection {Gutknecht_1997,
    AUTHOR = {Gutknecht, Martin H.},
     TITLE = {Lanczos-type solvers for nonsymmetric linear systems of
              equations},
 BOOKTITLE = {Acta numerica, 1997},
    SERIES = {Acta Numer.},
    VOLUME = {6},
     PAGES = {271--397},
 PUBLISHER = {Cambridge Univ. Press, Cambridge},
      YEAR = {1997},
   MRCLASS = {65F50 (65F10)},
  MRNUMBER = {1489258},
MRREVIEWER = {Marc Van Barel},
       DOI = {10.1017/S0962492900002737},
       URL = {https://doi.org/10.1017/S0962492900002737},
}

@incollection {Freund_Golub_Nachtigal_1992,
    AUTHOR = {Freund, Roland W. and Golub, Gene H. and Nachtigal, No\"{e}l M.},
     TITLE = {Iterative solution of linear systems},
 BOOKTITLE = {Acta numerica, 1992},
    SERIES = {Acta Numer.},
     PAGES = {57--100},
 PUBLISHER = {Cambridge Univ. Press, Cambridge},
      YEAR = {1992},
   MRCLASS = {65F10},
  MRNUMBER = {1165723},
MRREVIEWER = {L. Hageman},
       DOI = {10.1017/s0962492900002245},
       URL = {https://doi.org/10.1017/s0962492900002245},
}

@article {CroP17,
    AUTHOR = {Crouzeix, M. and Palencia, C.},
     TITLE = {The numerical range is a {$(1+\sqrt{2})$}-spectral set},
   JOURNAL = {SIAM J. Matrix Anal. Appl.},
  FJOURNAL = {SIAM Journal on Matrix Analysis and Applications},
    VOLUME = {38},
      YEAR = {2017},
    NUMBER = {2},
     PAGES = {649--655},
      ISSN = {0895-4798},
   MRCLASS = {47A25 (47A30)},
  MRNUMBER = {3666309},
MRREVIEWER = {Sophie Grivaux},
       DOI = {10.1137/17M1116672},
       URL = {https://doi.org/10.1137/17M1116672},
}

@article{HocL97,
	AUTHOR = {Hochbruck, Marlis and Lubich, Christian},
	TITLE = {On {K}rylov subspace approximations to the matrix exponential
	operator},
	JOURNAL = {SIAM J. Numer. Anal.},
	FJOURNAL = {SIAM Journal on Numerical Analysis},
	VOLUME = {34},
	YEAR = {1997},
	NUMBER = {5},
	PAGES = {1911--1925},
	ISSN = {0036-1429},
	MRCLASS = {65F30 (65L99)},
	MRNUMBER = {1472203},
	DOI = {10.1137/S0036142995280572},
	URL = {https://doi.org/10.1137/S0036142995280572},
	}

@Article {Philips_Shadid_Cyr_18,
	AUTHOR = {Phillips, Edward G. and Shadid, John N. and Cyr, Eric C.},
	TITLE = {Scalable preconditioners for structure preserving
	discretizations of {M}axwell equations in first order form},
	JOURNAL = {SIAM J. Sci. Comput.},
	FJOURNAL = {SIAM Journal on Scientific Computing},
	VOLUME = {40},
	YEAR = {2018},
	NUMBER = {3},
	PAGES = {B723--B742},
	ISSN = {1064-8275},
	MRCLASS = {65F08 (35Q61 65M22 65M60 78M25)},
	MRNUMBER = {3796378},
	MRREVIEWER = {S\'{e}bastien J. Boyaval},
	DOI = {10.1137/17M1135827},
	URL = {https://doi.org/10.1137/17M1135827},
}

@Article{Bonazzoli_Dolean_Graham_Spence_19,
	AUTHOR = {Bonazzoli, M. and Dolean, V. and Graham, I. G. and Spence, E.
	A. and Tournier, P.-H.},
	TITLE = {Domain decomposition preconditioning for the high-frequency
	time-harmonic {M}axwell equations with absorption},
	JOURNAL = {Math. Comp.},
	FJOURNAL = {Mathematics of Computation},
	VOLUME = {88},
	YEAR = {2019},
	NUMBER = {320},
	PAGES = {2559--2604},
	ISSN = {0025-5718},
	MRCLASS = {65N30 (35Q61 65F08 65F10 65N55 78A45)},
	MRNUMBER = {3985469},
	DOI = {10.1090/mcom/3447},
	URL = {https://doi.org/10.1090/mcom/3447},
}

@article{Lubomir_10,
	AUTHOR = {Ba\v{n}as, \v{L}ubom\'{\i}r},
	TITLE = {An efficient multigrid preconditioner for {M}axwell's
	equations in micromagnetism},
	JOURNAL = {Math. Comput. Simulation},
	FJOURNAL = {Mathematics and Computers in Simulation},
	VOLUME = {80},
	YEAR = {2010},
	NUMBER = {8},
	PAGES = {1657--1663},
	ISSN = {0378-4754},
	MRCLASS = {65M55 (65M60 78A25 78M10 82D40)},
	MRNUMBER = {2647259},
	DOI = {10.1016/j.matcom.2009.02.009},
	URL = {https://doi.org/10.1016/j.matcom.2009.02.009},
}

@article {Adler_Hu_Zik_17,
	AUTHOR = {Adler, J. H. and Hu, X. and Zikatanov, L. T.},
	TITLE = {Robust solvers for {M}axwell's equations with dissipative
	boundary conditions},
	JOURNAL = {SIAM J. Sci. Comput.},
	FJOURNAL = {SIAM Journal on Scientific Computing},
	VOLUME = {39},
	YEAR = {2017},
	NUMBER = {5},
	PAGES = {S3--S23},
	ISSN = {1064-8275},
	MRCLASS = {65M60 (35Q61 65M22 78M25)},
	MRNUMBER = {3716564},
	MRREVIEWER = {Bal\'{a}zs Kov\'{a}cs},
	DOI = {10.1137/16M1073339},
	URL = {https://doi.org/10.1137/16M1073339},
}

@article {Cyr_Gander_Thomas_07,
	AUTHOR = {St-Cyr, A. and Gander, M. J. and Thomas, S. J.},
	TITLE = {Optimized multiplicative, additive, and restricted additive
	{S}chwarz preconditioning},
	JOURNAL = {SIAM J. Sci. Comput.},
	FJOURNAL = {SIAM Journal on Scientific Computing},
	VOLUME = {29},
	YEAR = {2007},
	NUMBER = {6},
	PAGES = {2402--2425},
	ISSN = {1064-8275},
	MRCLASS = {65F10},
	MRNUMBER = {2357620},
	MRREVIEWER = {Vladimir B. Larin},
	DOI = {10.1137/060652610},
	URL = {https://doi.org/10.1137/060652610},
}

@article {Aruliah_Ascher_02,
	AUTHOR = {Aruliah, D. A. and Ascher, U. M.},
	TITLE = {Multigrid preconditioning for {K}rylov methods for
	time-harmonic {M}axwell's equations in three dimensions},
	JOURNAL = {SIAM J. Sci. Comput.},
	FJOURNAL = {SIAM Journal on Scientific Computing},
	VOLUME = {24},
	YEAR = {2002},
	NUMBER = {2},
	PAGES = {702--718},
	ISSN = {1064-8275},
	MRCLASS = {65N22 (65F10 65N55 78M99)},
	MRNUMBER = {1951063},
	MRREVIEWER = {Alessandro Veneziani},
	DOI = {10.1137/S1064827501387358},
	URL = {https://doi.org/10.1137/S1064827501387358},
}

@article {Hiptmair_99,
	AUTHOR = {Hiptmair, R.},
	TITLE = {Multigrid method for {M}axwell's equations},
	JOURNAL = {SIAM J. Numer. Anal.},
	FJOURNAL = {SIAM Journal on Numerical Analysis},
	VOLUME = {36},
	YEAR = {1999},
	NUMBER = {1},
	PAGES = {204--225},
	ISSN = {0036-1429},
	MRCLASS = {65N55 (65N30 78-08)},
	MRNUMBER = {1654571},
	MRREVIEWER = {Jacques Rappaz},
	DOI = {10.1137/S0036142997326203},
	URL = {https://doi.org/10.1137/S0036142997326203},
}
\end{document}